%% file: Nonuniqueness2d.tex
\documentclass[11pt,reqno]{amsproc}
\usepackage[margin=1in]{geometry}
\usepackage{amsmath, amsthm, amssymb}
\usepackage{times, esint,stackrel}
\usepackage{enumitem, graphicx}
\usepackage{color}
\usepackage{enumitem}   
\usepackage{xcolor} 
\usepackage{mathrsfs}
\usepackage{cite}
\usepackage{mathabx}
\usepackage{stackengine}
\usepackage{subcaption,graphicx}
\usepackage{float}
\usepackage{hyperref}
 \usepackage[capitalise]{cleveref}
\usepackage{autonum}
\usepackage{bbm}
\usepackage{tcolorbox}

\crefformat{equation}{(#2#1#3)}
\crefrangeformat{equation}{(#3#1#4) to~(#5#2#6)}
\crefmultiformat{equation}{(#2#1#3)}%
{ and~(#2#1#3)}{, (#2#1#3)}{ and~(#2#1#3)}

\usepackage{tikz}
\usepackage{pgfplots}
\pgfplotsset{compat=newest}
\usetikzlibrary{decorations.pathreplacing,decorations.markings}
\tikzset{
	on each segment/.style={
		decorate,
		decoration={
			show path construction,
			moveto code={},
			lineto code={
				\path [#1]
				(\tikzinputsegmentfirst) -- (\tikzinputsegmentlast);
			},
			curveto code={
				\path [#1] (\tikzinputsegmentfirst)
				.. controls
				(\tikzinputsegmentsupporta) and (\tikzinputsegmentsupportb)
				..
				(\tikzinputsegmentlast);
			},
			closepath code={
				\path [#1]
				(\tikzinputsegmentfirst) -- (\tikzinputsegmentlast);
			},
		},
	},
	mid arrow/.style={postaction={decorate,decoration={
				markings,
				mark=at position .5 with {\arrow[#1]{stealth}}
	}}},
}

\usepackage{mathtools}

\makeatletter

\def\@tocline#1#2#3#4#5#6#7{\relax
  \ifnum #1>\c@tocdepth 
  \else
    \par \addpenalty\@secpenalty\addvspace{#2}%
    \begingroup \hyphenpenalty\@M
    \@ifempty{#4}{%
      \@tempdima\csname r@tocindent\number#1\endcsname\relax
    }{%
      \@tempdima#4\relax
    }%
    \parindent\z@ \leftskip#3\relax \advance\leftskip\@tempdima\relax
    \rightskip\@pnumwidth plus4em \parfillskip-\@pnumwidth
    #5\leavevmode\hskip-\@tempdima
      \ifcase #1
       \or\or \hskip 1em \or \hskip 2em \else \hskip 3em \fi%
      #6\nobreak\relax
    \dotfill\hbox to\@pnumwidth{\@tocpagenum{#7}}\par
    \nobreak
    \endgroup
  \fi}
\makeatother

 \input{macros}
\title[Forced nonuniqueness in 2D Euler]{Self-similar instability and forced nonuniqueness: an application to the 2D Euler equations}
\author[M. Dolce]{Michele Dolce}
\address{Institute of Mathematics, EPFL, Station 8, 1015 Lausanne, Switzerland}
\email{michele.dolce@epfl.ch}

\author[G. Mescolini]{Giulia Mescolini}
\address{Institute of Mathematics, EPFL, Station 8, 1015 Lausanne, Switzerland}
\email{giulia.mescolini@epfl.ch}

\begin{document}

\begin{abstract}
Building on an approach introduced by Golovkin in the '60s, we show that nonuniqueness in some forced PDEs is a direct consequence of the existence of a self-similar linearly unstable eigenvalue: the key point is a clever choice of the forcing term removing complicated nonlinear interactions. We use this method to give a short and self-contained proof of nonuniqueness in 2D perfect fluids, first obtained in Vishik's groundbreaking result.
 In particular, we present a direct construction of a forced self-similar unstable vortex, where we treat perturbatively the self-similar operator  in a new and more quantitative way. 
\end{abstract}

\maketitle

\section{Introduction}
Many partial differential equations arising from the description of physical systems, such as the Euler, Navier-Stokes or Boussinesq equations, when seen in self-similar variables $(\xi,t)=(x/t^{\gamma},t)$ for some $\gamma>0$ give rise to a Cauchy problem of the form 
\begin{equation}
\label{eq:genpde}
    \begin{cases}
        t\de_t G=L(G)+B(G,G)+F, \qquad t\geq 0, \quad \xi\in \RR^d,\\
        G|_{t=0}=g_0.
        \end{cases}
\end{equation}
Here $G:[0,+\infty)\times \RR^d\to \RR^m$ is the unknown self-similar profile, $L:X\to \mathbb{R}^m$ is a linear operator with $X$ being a Banach space, $B:X\times X\to \mathbb{R}^m$ a bilinear one, and $F:[0,+\infty)\times \RR^d\to \RR^m$ is a forcing term. We claim that, by choosing appropriately the forcing term, it is straightforward to generate two solutions emanating from the same initial data $g_0$ if the linearization of the PDE at $g_0$ has an unstable eigenvalue. This  observation is a consequence of a nice trick used by Golovkin\footnote{He was a student of Lady\v{z}enskaja \cite{ladyzhenskaya1971kirill} passed away prematurely in the 1969. He was interested in the study of nonuniqueness of stationary states in fluids with given boundary conditions.} \cite{golovkin1964nonuniqueness,golovkin1965examples} and the self-similar non-uniqueness scenario proposed by Jia and \v{S}ver\'{a}k in the 3D Navier-Stokes equations \cite{jia2015incompressible}. The ideas of Golovkin were used in the example of nonuniqueness proposed by  Lady\v{z}enskaja  \cite{ladyvzenskaja1969example} and its interplay with an approximately self-similar scenario was explored for some shell-models in \cite{filonov2021nonuniqueness}. We will refer to this method as \textit{Golovkin's trick} and it can be explained as follows: we look for two solutions in the form
\begin{equation}
   \label{eq:gpm}
    G^\pm(\xi,t)=g_0(\xi)\pm \mathrm{Re}(t^\lambda \eta(\xi)),
\end{equation}
with $\eta:\mathbb{R}^d\to \mathbb{C}^m$ and $\lambda\in \mathbb{C}$ with $\Re(\lambda)>0$. This last assumption guarantees that $G^\pm|_{t=0}=g_0$ if $\eta$ is sufficiently regular. Then, we still have the freedom of choosing the forcing term $F$ and the equation satisfied by $\eta$. Plugging the ansatz \cref{eq:gpm} in \cref{eq:genpde}, we find that 
\begin{equation}
\mp\mathrm{Re}\left( t^\lambda(L(\eta)+B(g_0,\eta)+B(\eta,g_0)-\lambda \eta )\right)=L(g_0)+B(g_0,g_0)+B(\mathrm{Re}(t^\lambda\eta),\mathrm{Re}(t^\lambda\eta))+F.
\end{equation}
Hence, if we define 
\begin{align}
&F\coloneqq-L(g_0)-B(g_0,g_0)-B(\mathrm{Re}(t^\lambda\eta),\mathrm{Re}(t^\lambda\eta))\\
&\mathcal{L}_{g_0}(\eta)\coloneqq L(\eta)+B(g_0,\eta)+B(\eta,g_0),
\end{align}
we see that \cref{eq:gpm} solves \cref{eq:genpde} if there exists a couple $(\lambda,\eta)$ with $\Re(\lambda)>0$ and a function $\eta$ such that
\begin{equation}
    \mathcal{L}_{g_0}(\eta)=\lambda \eta.
\end{equation}
Namely, $\eta$ is an eigenfunction associated to any unstable eigenvalue of the operator $\mathcal{L}_{g_0}$. With this trick, which heavily relies on the fact that the nonlinearity is quadratic, one can always generate exactly two different solutions (if $\eta$ exists and it is non-trivial) with the force $F$ (that depends on $g_0$ and $\eta$). The two different solutions are \[
G^+(\xi,t)=g_0(\xi)+\Re(t^\lambda \eta(\xi)) \quad \text{and} \quad G^-(\xi,t)=g_0(\xi)-\Re(t^\lambda \eta(\xi)).\]

Notice that $g_0$ does not even need to solve the unforced PDE. However, it is desirable that the force $F$ satisfies at least some integrability properties, which are directly related to properties of  $g_0$ and the eigenfunction $\eta$. 
A natural question arises:
\begin{quote}
    \textbf{Question:} Under which conditions on $L,B$ one can always find a smooth profile $g_0$ so that the operator $\mathcal{L}_{g_0}$ has at least one unstable eigenvalue with a smooth eigenfunction?
\end{quote}
We require some (unspecified) smoothness to guarantee that the self-similar solution lives in the natural $L^p$ spaces where the existence of weak solutions is known. Finally, notice that the property of not having eigenvalues with positive real part for \emph{all} smooth $g_0$ seems to impose a quite rigid structure on the operators $L$ and $B$ (such as $L$ being a good enough dissipative operator  or $\mathcal{L}_{g_0}$ being always skew-adjoint for any $g_0$).



\subsection{Forcing the nonuniqueness in 2D perfect fluids}
\label{sec:vishik}
Spectacular examples of nonuniqueness with a force have been recently obtained for equations arising in fluid dynamics \cite{jia2015incompressible,vishik_notes,castro2024proof,albritton23Leray,albritton23Hypo,albritton22annals}. The PDEs involved are of the form \cref{eq:genpde} when seen in self-similar variables. In this paper, we only focus on the specific example of the 2D forced and incompressible Euler equations. In vorticity formulation they are given by
\begin{align}
\label{eq:Euler}
\begin{cases}\de_t \omega +u\cdot \nabla \omega=f, \qquad x\in \RR^2, \, t\in [0,T),\\
u=\nabla^\perp\psi, \quad \Delta \psi=\omega,\\
\omega|_{t=0}=\omega_0,
\end{cases}
\end{align}
where for our discussion $T>0$ will always be finite.
It is well known that for $\omega_0\in L^1_x\cap L^\infty_x$ and $f\in L^1_t(L^1_x\cap L^\infty_x)$, Yudovich's theory \cite{yudovich1963non} guarantees the existence and uniqueness of  weak solutions to \cref{eq:Euler}. The outstanding open problem of nonuniqueness \emph{without a force} with some integrable vorticity has been solved only recently by Brué, Colombo and Kumar \cite{brue2024flexibility}. This result relies on a new convex integration scheme that works for $\omega_0\in L^{1+\eps}_x(\TT^2)$ with $\eps$ sufficiently small. However, nonuniqueness with convex integration  can only be obtained  for $\omega_0 \in L^p_x(\TT^2)$ with $p<3/2$ in view of energy conservation for larger values of $p$. For general $p$ without a force, the problem is fully open and here we just mention two promising scenarios proposed by Bressan and Shen \cite{bressan21shen} (see also \cite{bressan20murray}) and Elgindi, Murray and Said \cite{elgindi2023wellposedness}. 
If one instead uses the extra degree of freedom given by a force, then it is possible to achieve a large range of $L^p_x$ spaces not covered by Yudovich's theory. This was first proved in Vishik's groundbreaking results  \cite{vishik2018instability1,vishik2018instability2}, establishing the nonuniqueness of solutions ``emanating'' from a force  $f\in L^\infty_t(L^1_x\cap L^p_x)$.  Vishik's result was later revisited and refined in \cite{vishik_notes}, see also \cite{castro2024proof}, and the theorem can now be stated as follows.
\begin{theorem}[\cite{vishik_notes}]
\label{th:vishik}
For every $p\in(2,\infty)$ there is a triple $\omega_0\in L^1_x\cap L^p_x$, $u_0\in L^2_x$ and $f\in L^1_t(L^1_x\cap L^p_x)$ with the property that there are uncountably many solutions $\omega\in L^\infty_t(L^1_x\cap L^p_x)$ and $u\in L^\infty_t(L^2_x)$ of \cref{eq:Euler}. Moreover, $\omega_0$ can be chosen to vanish identically.
\end{theorem}
To explain the main ideas behind the proof of Theorem \ref{th:vishik}, we  introduce the self-similar variables 
\begin{align}
\xi=\frac{x}{t^{\frac{1}{\alpha}}}, \qquad (\omega,\psi,f)(x,t)=\frac{1}{t}(\Omega,\, t^{\frac{2}{\alpha}}\Psi,\, t^{-1}F)(\xi,t),
\end{align}
where $\alpha \in (0,2)$ is a free parameter due to the scaling properties of \cref{eq:Euler}.
The equations \cref{eq:Euler} in self-similar variables are
\begin{align}
\label{eq:ssEuler}
\begin{cases}\displaystyle t\de_t\Omega -\big(1+\frac{\xi}{\alpha}\cdot \nabla\big)\Omega+U \cdot \nabla \Omega=F, \qquad \xi \in \RR^2, \, t>0,\\
U=\mathrm{BS}(\Omega)=\nabla^\perp \Psi, \qquad 
\Delta \Psi=\Omega.
\end{cases}
\end{align}
The $t\to0$ limit is a singular limit, but one can carefully assign  an initial data in \cref{eq:ssEuler}, and then recover the initial data $\omega_0$ in $L^{p}$ sense for instance. 

The first idea behind the proof of Theorem \ref{th:vishik} is to exploit Jia and \v{S}ver\'{a}k strategy \cite{jia2015incompressible}: find a linearly unstable self-similar profile (ideally without a force). The simplest exact self-similar solution is the radial vortex $\Omega=|\xi|^{-\alpha}$, which unfortunately cannot be used for the desired nonuniqueness as we explain in more detail in \cref{rem:ralpha}. By perturbing this radial vortex, self-similar spirals have been recently constructed in \cite{elling12spiral,elling16spiral,shao2023self}. However, the study of spectral properties of these exact self-similar solutions remains an extremely interesting open problem to date.
To overcome the issue of using an exact self-similar solution, the first key observation by Vishik is that any radial vortex is a self-similar solution with a suitable force. Indeed, a solution to \cref{eq:ssEuler} is 
\begin{equation}
    \label{eq:VishikVortex}
\Omega_V(|x|)\coloneqq\beta g(|x|), \qquad U_V(x)=\beta v(|x|)x^\perp,\qquad F_V=-\beta(1+\xi/\alpha)\cdot \nabla g,
\end{equation}
with $0\neq \beta\in \RR$, $g:\RR\to \RR$ a given smooth function and 
\begin{equation}
    \label{def:v}
    v(r)=\frac{1}{r^2}\int_0^r\rho g(\rho)\dd \rho.
\end{equation}
However, it was not clear if these type of solutions can be unstable in the self-similar framework, but here comes the second fundamental observation by Vishik: consider the linearized operator
  \begin{equation}
  \label{def:Lg}
      \mathcal{L}_{\beta,g}(w)\coloneqq\frac{1}{\beta}(1+\frac{\xi}{\alpha})\cdot \nabla w-U_V\cdot \nabla w-\mathrm{BS}(w)\cdot \nabla g,
  \end{equation}
  arising from the linearization of \cref{eq:ssEuler} at $g$. Then, the operator $\mathcal{L}_{\beta,g}$ can be seen as a perturbation of the one obtained by formally setting $\beta=\infty$. But $\mathcal{L}_{\infty,g}$ is nothing more than the linearized operator associated to the 2D Euler equations in standard variables. For the latter, examples of spectrally unstable vortices were rigorously proved to exist already in the work of Z. Lin \cite{lin2003instability}, for instance. However, the spectral perturbation in $\beta$ is quite delicate and in fact the structure of the specific operators involved is crucial (indeed, with a naive approach one could encounter derivative losses).
More precisely, thanks to the refinements of Vishik's results obtained in \cite{vishik_notes,castro2024proof}, the following is now established.
\begin{theorem}
\label{th:unstable}
For a suitable class of $g\in C^{\infty}_c(\RR)$ with zero mean $\int_{\RR}g(r)r\dd r=0$, there exists $\beta_0>0$ such that for all $\beta>\beta_0$ there exists $(\lambda,w)$ with $\mathrm{Re}(\lambda)>0$ and $w\in C^\infty_c(\RR^2\setminus \{0\})\cap C^2_c(\RR^2)$, $\int_{\RR^2}w(x)\dd x=0$ such that
\begin{equation}
    \mathcal{L}_{\beta,g}(w)=\lambda w.
\end{equation}   
\end{theorem}
These smoothness properties of the eigenfunction and a simplified construction of a compactly supported unstable vortex have been established in \cite{castro2024proof}, where the zero mean condition was imposed to simplify the proof to get $L^2_x$ velocity fields.
The proposition above implies that $\Omega=g+\Re(t^\lambda \eta)$ solves
\begin{equation}
    t\de_t \Omega=\mathcal{L}_{\beta,g}(\Omega)+F_V,
\end{equation}
which is the linearized problem of \cref{eq:ssEuler} at $g$. We thus have two solutions of the linearized problem emanating from the same initial data, namely $\Omega_0=g$ and $\Omega_1=g+\Re(t^\lambda w)$. To obtain two solutions in the nonlinear problem, one has to perform a delicate nonlinear argument in order to  correct the error arising when plugging the linearized solution in \cref{eq:ssEuler}, which is the last key point in the proof of \cref{th:vishik}.

\subsection{Golovkin's trick in the 2D Euler equations}
\label{sec:golovkin_euler}
Combinining \cref{th:unstable} with the Golovkin's trick, we observe that the nonlinear argument mentioned before can be avoided if one is only interested in producing \emph{exactly two} solutions.
Indeed, it is enough to define 
\begin{equation}
\label{def:FG}
    F=-\beta(1+\xi/\alpha)\cdot \nabla g+\mathrm{BS}(\Re(t^\lambda w))\cdot \nabla \Re(t^\lambda w),
\end{equation}
where $g$ is an unstable vortex and $(\lambda,w)$ is the eigenpair in Theorem \ref{th:unstable}. Then, we have the following two solutions to \cref{eq:ssEuler}
\begin{equation}
\label{eq:Ompm}
    \Omega^+=g+\Re(t^\lambda w), \qquad \Omega^-=g-\Re(t^\lambda w).
\end{equation}
The associated solutions in original variables are $\omega^\pm(x,t)=t^{-1}\Omega^\pm(x/t^{1/\alpha})$. Thanks to the properties of $g$ and $w$, for any $T>0$ we also know that 
\begin{equation}
    \|\omega^\pm(t)\|_{L^{\frac{2}{\alpha}}}=\|\Omega^\pm(t)\|_{L^{\frac{2}{\alpha}}}, \qquad  \|\omega^\pm(t)\|_{L^p}<\infty \qquad \text{for all } 0\leq t<T \text{ and } p<\tfrac{2}{\alpha}.
\end{equation}
Moreover $\omega^\pm \to 0$ as $t\to 0$ in $L^p$ with $p<2/\alpha$ but $\omega^+(t)\neq\omega^{-}(t) \neq 0$ in $L^p$ for all $t>0.$

It remains to verify that the velocity fields $u^{\pm}(t,\cdot)=\mathrm{BS}(\omega^{\pm})(t,\cdot)\in L^2_x$. To this end, it is enough to observe $\int_{\RR^2}\omega^\pm \dd x=0$ since $g(r)$ is a zero mean vortex and $w$ has zero mean (it is in fact a function concentrated in an angular Fourier mode $k\neq 0$.)
Therefore, the two solutions in \cref{eq:Ompm} enjoy the same properties of the infinitely many ones found in \cite{vishik2018instability1,vishik2018instability2,vishik_notes,castro2024proof}. Thus,  with the Golovkin's trick, nonuniqueness becomes a direct corollary of Theorem \ref{th:unstable}.

\begin{remark}
The nonuniqueness statement obtained with the Golovkin's trick is  weaker compared to  what is obtained in \cite{vishik2018instability1,vishik2018instability2,vishik_notes,castro2024proof}. One has exactly two solutions and not a whole family living in the unstable manifold. The main difference is in the choice of the force. In our case, the force depends on the eigenpair $(\lambda,w)$. Therefore, once we fix $(\lambda,w)$ we find exactly two solutions with a single force. On the other hand, the results in \cite{vishik2018instability1,vishik2018instability2,vishik_notes,castro2024proof} show that with the force $F_V$ (which only depends on the vortex $g$), one can construct infinitely many solutions due to the possibility of suitably scaling any eigenpair $(\lambda,w)$ (or changing the $m$-fold symmetry of $w$). Therefore, the nonuniqueness found in \cite{vishik2018instability1,vishik2018instability2,vishik_notes,castro2024proof} is much more dramatic. The price to pay is a delicate nonlinear argument needed to correct the nonlinear error arising from the eigenpair (the last term in \cref{def:FG}).
\end{remark}
\begin{remark}[On the stability of a power law vortex]
\label{rem:ralpha}
Consider the power law vortex $\omega_\alpha(x)=(2-\alpha)|x|^{-\alpha}$ with associated self-similar profile $\Omega_\alpha(\xi)=(2-\alpha)|\xi|^{-\alpha}$ with $\alpha\in(0,2)$, the first solves \cref{eq:Euler} and the second \cref{eq:ssEuler}. We exhibit a simple example of nonuniqueness arising from modes $|k|=1$. In this case, one can explicitly find nonunique solutions for the linearized problem around $\omega_\alpha$ in standard variables, that is 
    \begin{equation}
\label{eq:linraplpha}\begin{cases}\de_t w_{\alpha}+r^{-\alpha}\de_\theta w_{\alpha}-\alpha(2-\alpha) r^{-(2+\alpha)}\de_\theta \varphi_{\alpha}=0,\\
       (\de_{rr}+r^{-1}\de_r+r^{-2}\de_{\theta\theta})\varphi_\alpha=w_\alpha.
        \end{cases}
    \end{equation}Imposing a separation of variables ansatz $w_\alpha(r,\theta,t)=t^\gamma r^q\cos(k\theta)$, it is not hard to verify that setting $q=-(1+\alpha)$, $k=1$ and $\varphi_\alpha(r,\theta,t)=-\frac{1}{\alpha(2-\alpha)} t^\gamma(r^{1-\alpha}\cos(\theta)+\frac{\gamma}{t} r\sin(\theta) )$ then we explicitly solve \cref{eq:linraplpha} for any $\gamma >0 $. This explicit solution emanates from the radially symmetric vortex and exact self-similar solution $\omega_\alpha$ and it is instantaneously breaking the radial symmetry. The factor $\gamma$ can be interpreted as an unstable  eigenvalue for the self-similar problem, but the freedom in choosing $\gamma$ relies on special cancellations for the modes $k=\pm1$ (as $r\sin(\theta)$ being harmonic). Moreover,  applying the Golovkin's trick we know that $\omega^\pm(r,\theta,t)=(2-\alpha)r^{-\alpha}\pm w_\alpha(r,\theta,t)\in L^\infty_t(L^1_{{\rm loc},x}\cap L^{p_\alpha}_{{\rm loc},x}(\RR^2))$ are two solutions to  \cref{eq:Euler} with the force
    \begin{equation}
        f=-\frac{1}{\alpha(2-\alpha)}\nabla^\perp \varphi_\alpha\cdot \nabla w_\alpha \in L^1_{t}\big((L^1_{{\rm loc},x}\cap L^{\tilde{p}_\alpha}_{{\rm loc},x})(\RR^2\setminus \{0\})\big),
    \end{equation}
    where $p_\alpha,\tilde{p}_\alpha\in (1,\infty)$ can be explicitly computed. 
However, $f\notin L^1_t(L^p_{{\rm loc},x}(\RR^2))$ for any $1\leq p\leq \infty$, meaning that this example does not prove the desired nonuniqueness for the 2D forced Euler equations.
Finally, let us also remark that linear stability properties of  $\Omega_\alpha$ have been rigorously studied in \cite{coiculescu2024stability}, exhibiting a concept of stability that naturally excludes the pathological behavior described above.
\end{remark}

\subsection{Main result}
The main goal of this paper is to provide a self-contained proof of the following.
\begin{theorem}
\label{th:vishik2}
For every $p\in(2,\infty)$ there is a triple $\omega_0\in L^1_x\cap L^p_x$, $u_0\in L^2_x$ and $f\in L^1_t(L^1_x\cap L^p_x)$ with the property that there are two solutions $\omega^+,\omega^{-}\in L^\infty_t(L^1_x\cap L^p_x)$ and $u^+,u^-\in L^\infty_t(L^2_x)$ of \cref{eq:Euler}. Moreover, $\omega_0$ can be chosen to vanish identically.
\end{theorem}
 Thanks to the Golovkin's trick, we know that everything is reduced in obtaining yet another proof of Theorem \ref{th:unstable}.  The novelty is that we construct the eigenpair $(\lambda,\eta)$ directly in self-similar variables with a compactly supported vortex $g\in C^\infty_c(\RR)$. We thus avoid performing the spectral perturbation argument after the construction of an unstable vortex in standard variables as in \cite{vishik2018instability1,vishik2018instability2,vishik_notes,castro2024proof}.

We will still treat the operator arising from the self-similar change of variables perturbatively, but the new key observation is in how we include it directly in the construction of the eigenpair. To do this, we first observe that when studying spectral stability properties of vortices (or shear flows) in standard variables, one formally divide the Rayleigh equation by $(ik v(r)-\lambda)$, where $v$ is as in \cref{eq:VishikVortex}. This operation is related to considering the resolvent of the operator $U_V\cdot \nabla$. Here, we will instead consider also the self-similar part and study the resolvent of the operator 
\begin{equation}
\label{eq:calSbeta}
   \mathcal{S}_\beta= U_V\cdot \nabla -\frac{1}{\alpha \beta}(\xi \cdot \nabla +1),
\end{equation}
which takes into account the skew adjoint part of $\frac{1}{\alpha\beta}\xi\cdot \nabla$.
In particular, we perform an expansion for $\beta\gg1$ and we obtain the following.
\begin{proposition}
\label{th:resolvent}
For all $z\in \mathbb{C}$ with $\mathrm{Re}(z)>0$, there exists $\beta_0>0$ such that for all $\beta\geq \beta_0$ the following holds true: there exists an operator $\mathcal{T}_{z,\beta} \coloneqq \mathcal{T}_\beta :D(\mathcal{S}_\beta) \to L^2_x$ such that    \begin{equation}
        (\mathcal{S_\beta}-z)^{-1}=(U_V\cdot \nabla -z)^{-1}-\mathcal{T}_\beta
    \end{equation}
    where $\|\mathcal{T}_{\beta}\|_{X\to X}\to 0$ as $\beta \to \infty$ with $X\in \{L^2_x,\dot{H}^1_x\}$.
\end{proposition}
With this result, we have a quantitative way of expanding the resolvent of $\mathcal{S}_\beta$ so that we can follow the standard construction of unstable eigenvalues in the 2D Euler equations. To construct the eigenpair, we will combine the arguments used in \cite{vishik_notes,castro2024proof,kumar2023simple,albritton2023linear}. Namely, we use a compactly supported vortex as in \cite{castro2024proof}. This allow us to perform an inner-outer splitting as in \cite{kumar2023simple,albritton2023linear} and follow the construction from a neutral limiting mode as done in \cite{vishik_notes,kumar2023simple,lin2003instability}. Moreover, with a compactly supported background vortex it is rather straightforward to show the desired properties of the eigenfunction, see \cref{cor:eta}.

\medskip 

The rest of the paper is organized as follows: in 
 \cref{sec:instability} we present our direct proof of \cref{th:unstable}. In \cref{sec:proof_inverse} we show the proof of a result that implies \cref{th:resolvent}, which is a key ingredient in the proof of \cref{th:unstable}.

\section{A direct approach to self-similar instability}
\label{sec:instability}
The goal of this section is to provide a direct proof of Theorem \ref{th:unstable}. We will follow the standard strategy of constructing the unstable eigenpair as the one arising from a \textit{neutral limiting mode}, a method that dates back at least to the work of Tollmien \cite{tollmien1935allgemeines}. This was first rigorously implemented by Z. Lin \cite{lin2003instability} and improved and refined in \cite{vishik_notes}. The main novelty compared to the recent works \cite{vishik_notes, vishik2018instability1,vishik2018instability2, castro2024proof} is in how we treat the self-similar operator. Here, we directly include it in the argument to construct the eigenvalue, thus avoiding the use of general spectral perturbation arguments.

\subsection{The self-similar Rayleigh's equation} We start our analysis by writing the system for the perturbation around a vortex as in \cref{eq:VishikVortex}. To this end, let
$$\Omega^\pm=\Omega_V\pm W.$$ 
To write the system for $W$, we introduce the standard polar coordinates $\xi=r(\cos(\theta),\sin(\theta))$ where $(r,\theta)\in \RR_+\times \TT$. The gradient and the Laplacian are  $\nabla_{r,\theta}=(\de_r,r^{-1}\de_\theta)$ and $\Delta_{r,\theta}=\de_{rr}+r^{-1}\de_r+r^{-2}\de_{\theta \theta}$ respectively. Then, it is not hard to check that if $W$ satisfies the following system 
\begin{equation}
    \label{eq:W}
\begin{cases}
    t\de_tW-(1+\tfrac{1}{\alpha}r\de_r)W+\beta v(r)\de_\theta W-\beta\tfrac{g'(r)}{r}\de_\theta \Phi=0,\\
    \Delta_{r,\theta} \Phi=W, \quad \Phi \to 0 \text{ as } r\to 0^+ \text{ and } r\to +\infty,\\
    W|_{t=0}=0, 
\end{cases}
\end{equation}
then $\Omega^+$ and $\Omega^-$ are two solutions of \cref{eq:ssEuler} with the force 
\begin{equation}
\label{eq:F}
F=\nabla^\perp_{r,\theta} \Phi\cdot \nabla_{r,\theta} W-\beta(1+\tfrac{1}{\alpha}r
\de_r)g(r).
\end{equation}
We aim at solving \cref{eq:W} in $L^2(r\dd r \dd \theta)$ based spaces. Therefore, given the structure of the equations, it is convenient to make the following ansatz \begin{align}
\label{eq:WPhi}
    W(t,r,\theta)=\mathrm{Re}(t^{\beta \lambda}r^{-\frac12}\e^{ik\theta}\eta(r)), \qquad \Phi(t,r,\theta)=\mathrm{Re}(t^{\beta \lambda}r^{-\frac12}\e^{ik\theta}\phi(r)),
\end{align}
where $k\in \ZZ$ such that $|k|\geq 2$, $\beta\gg 1$ to be specified later and $(\lambda,\eta)\in \{z\in \mathbb{C}\, :\,\Re(z)>0\}\times H^1(\RR)$ is the eigenpair we are looking for. More precisely, the eigenfunction  in \cref{th:unstable} is $w=r^{-\frac12}\e^{ik\theta}\eta(r)$. The factor $r^{-\frac12}$ is introduced to work in $L^2(\dd r)$ based spaces instead of $L^2(r\dd r).$ We first observe that $\Delta_{r,\theta}\Phi=W$ becomes 
\begin{equation}
\rmL_k(\phi)\coloneqq\phi''-\frac{1}{r^2}\big(k^2-\frac14\big)\phi=\eta, \qquad \phi\to 0  \text{ as } r\to 0^+ \text{ and } r\to \infty.
\end{equation}
Notice that $\rmL_k$ is  the conjugation of the Laplacian restricted to the $k$-th angular Fourier mode with the weight $r^{-\frac12}$ , namely $\rmL_k=r^\frac12(\de_{rr}+r^{-1}\de_r-r^2/k^2)r^{-\frac12}.$
Then, rewriting the first equation in \cref{eq:W}  in terms of $\phi$, and writing $$
r^{\frac12}(1+\frac{1}{\alpha}r\de_r)r^{-\frac12}=1-\frac{1}{\alpha}+\frac{1}{\alpha}(r\de_r+\frac12)$$
we arrive at the 
\emph{self-similar Rayleigh's equation}
\begin{equation}
\label{eq:rayleigh_phi}
\left(ikv(r) - \frac{1}{\alpha\beta} \big(r\partial_r + \frac{1}{2}\big) +\lambda_\beta\right)\rmL_k(\phi)-ik\frac{g'(r)}{r}\phi=0,
\end{equation}
where we denote 
\begin{equation}
    \label{def:lambdabeta}    \lambda_\beta\coloneqq\lambda-\frac{\alpha-1}{\alpha\beta}.
\end{equation}
Our first goal is to find a solution $(\phi,\lambda_\beta) \in \dot{H^1}(\dd r)\times \mathbb{C}$ to \cref{eq:rayleigh_phi} with $\Re(\lambda_\beta)>0$ for a class of possibly unstable vortices.  By Rayleigh's criterion \cite{marchioropulvirenti}, we recall that $g$ needs to have at least one non-degenerate critical point. It is also convenient to assume it has compact support and nice properties close to the origin.  We then define the following class:

\begin{definition}
\label{def:class} Our class $\mathfrak{G}$ of possibly unstable compactly supported and smooth vortices consists of functions $g:\RR\to \RR$ such that:
\begin{itemize}
    \item[i)] $g\in C^\infty_c(\RR)$ and $\mathrm{supp}(g)\subseteq[0,r_0]$ for a finite $r_0>0$.
    \item[ii)] $g(r)=g_0-\gamma_0r^2$ for $r\in [0,\delta_0]$ for some $\delta_0\in (0,1)$ and $g_0,\gamma_0>0$ such that $g(r)>0$ on $[0,\delta_0]$.
    \item[iii)] $g(r)=-g_1+\tfrac12\gamma_1(r-r_1)^2$ for $r\in[r_1-\delta_1,r_1+\delta_1]$ with $g_1>0$ and some $\delta_1\in (0,1)$ and $\gamma_1>0$; in particular, $g'(r_1)=0$. 
    \item[iv)]  $g'(r)\leq0 $ in $[0,r_1]$ and $g'(r)\geq 0$ on $(r_1,r_0)$.  
   \item[v)] The vorticity profile has zero mean, namely $\int_{\RR}g(r)r \dd r=0$.
    \end{itemize}
\end{definition}

\begin{figure}[H]

\begin{tikzpicture}
  \begin{axis}[
    axis lines=middle,
    xlabel={$r$},
    ylabel={$g(r)$},
    xmin=-1, xmax=9,
    ymin=-4, ymax=8,
    xtick=\empty,
    ytick=\empty,
    samples=200,
    domain=-5:20,
    axis line style={->},
    extra x ticks={3.957,6.264},
  extra x tick labels={, },
    extra x tick style={
    tick style={black, thick}
  }
  ]
    \draw[dashed] (3.957,0)--(3.957,-3);
\node[black, label=above:{$r_1$}] at (axis cs:3.957,0) {};
\node[black, label=above:{$r_0$}] at (axis cs:6.264,0) {};

  \addplot[cyan!70!blue, thick, domain=0:2.2439] {5 *cos(deg(0.7 * x))};
  \addplot[cyan!70!blue, thick, domain=2.243:5.23] {1.021*(x-3.957)^(2)-3};
  \addplot[cyan!70!blue, thick, domain=5.23:6.264] {-1.258*(x-6.264)^(2)};
  \addplot[cyan!70!blue, thick, domain=6.264:9] {0*x};

  \end{axis}
\end{tikzpicture}
\caption{Qualitative plot of a vortex profile $g \in \mathfrak{G}$.}
\end{figure}

\begin{remark}
    The hypothesis v) is needed to guarantee that the velocity field associated to the vortex with vorticity profile $g$ is in $L^2_x$. This was already used in \cite{castro2024proof} to simplify the proof of \cref{th:vishik2}, that would otherwise require additional technical steps to ensure $u^{\pm}\in L^\infty_tL^2_x$. In fact, the vortex used in \cite{castro2024proof} belongs to the class $\mathfrak{G}.$
\end{remark}
\noindent 
For a vortex in this class, we are able to prove the following.
\begin{theorem}
\label{th:phiH1}
    There exists a vortex $g$ in the class $\mathfrak{G}$, $0<\eps<1$  and $\beta_0>0$ such that the following holds true: for all $\beta\geq \beta_0$, there exists $(\phi,\lambda_\beta)\in \dot{H}^1(\RR)\times \mathbb{C} $ solving \cref{eq:rayleigh_phi} with $\Re(\lambda_\beta)>\eps/2$.
\end{theorem}
The proof of this theorem occupies the rest of this section and it relies on a new approach to handle the self-similar operator, as it will be explained in the sequel. For now, let us just observe that with this theorem we are able to solve \cref{eq:rayleigh_phi} also in vorticity formulation, namely when identifying $\rmL_k(\phi)$ with $\eta$. In particular, we have the following.
\begin{corollary}
\label{cor:eta}
Let $(\phi,\lambda_\beta)$ be the couple solving \cref{eq:rayleigh_phi} found in Theorem \ref{th:phiH1}. Define 
\begin{equation}
\label{def:IC}
I(r)=r^{\alpha\beta\lambda_\beta-\frac12}\e^{ik\alpha \beta \big(v(0)\log(r)+\int_0^rs^{-1}(v(s)-v(0))\dd s\big)}, \qquad  C_\phi=ik\alpha\beta \int_0^{r_0}\frac{g'(s)}{s^2I(s)} \phi(s)\dd s.
\end{equation}
Then, $\eta=\rmL_k(\phi)$ defined as 
\begin{equation}
\label{def:eta}
\eta(r)=I(r)\big(C_\phi-ik\alpha\beta \int_0^r\frac{g'(s)}{s^2I(s)} \phi(s)\dd s\big)
 \end{equation}
solves \cref{eq:rayleigh_phi} and $\eta\in C^\infty_c(\RR\setminus \{0\})\cap C^\gamma_c(\RR)$ for $\gamma=\alpha \beta\Re(\lambda_\beta)-1/2$. In particular, for $\beta$ sufficiently large $\eta\in C^2_c(\RR)$.
\end{corollary}
Notice that all the limitations in the smoothness are related to $I(r)$ which is singular only at the origin. Indeed, for the standard Rayleigh's equation, i.e. $\beta=+\infty$, we would get a smooth solution by standard elliptic arguments. 
Let us present the proof of this corollary which would then readily imply the proof of the nonuniqueness theorem.
\begin{proof}[Proof of Corollary \ref{cor:eta}]
    We can first  rewrite \cref{eq:rayleigh_phi} as 
\begin{equation}
    r\de_r \eta=\big(-\frac12+\alpha\beta(\lambda_\beta+ikv(r))\big)\eta-ik\alpha\beta\frac{g'(r)}{r}\phi.
\end{equation}
Now, one would be tempted to take the integrating factor naturally associated with the equation above with $\phi=0$. However, this would require to integrate $v(r)/r$ and we cannot do so for $r\to 0$. To overcome this issue, we first make a change in ``phase" by considering \begin{equation}
    \widetilde{\eta}(r)\coloneqq \e^{-ik\alpha\beta v(0)\log(r)}\eta(r),
\end{equation}
where $v(0)=g_0/2$ for vortices in class $\mathfrak{G}$ as in \cref{def:class}. A direct computation gives 
\begin{equation}
    r\de_r\widetilde{\eta}=\big(-\frac12+\alpha\beta(\lambda_\beta+ik(v(r)-v(0))\big)\eta-ik\alpha\beta\frac{g'(r)}{r}\e^{-ik\alpha\beta v(0)\log(r)}\phi.
\end{equation}
We now denote
\begin{equation}
    \widetilde{I}(r)=r^{\alpha\beta\lambda_\beta-\frac12}\e^{ik\alpha \beta \int_0^rs^{-1}(v(s)-v(0))\dd s}=I(r)\e^{-ik\alpha\beta v(0)\log(r)}.
\end{equation}
Notice that the integral in the definition of $\widetilde{I}$ is well defined by the properties of $g$ and hence $v$ as in \cref{def:v} and we have $$(\widetilde{I}^{-1}\widetilde{\eta})(r)=(I^{-1}\eta)(r).$$ Then, by the definitions of $\widetilde{\eta},\widetilde{I}$ and $I$, we get
\begin{equation}
r\de_r(\widetilde{I}^{-1}\widetilde{\eta})=r\de_r(I^{-1}\eta)=-ik\alpha \beta \frac{g'(r)}{r\widetilde{I}(r)\e^{ik\alpha\beta v(0)\log(r)}}\phi=-ik\alpha \beta \frac{g'(r)}{rI(r)}\phi.
\end{equation}
 Now, since $g'(r)=0$ for $r\geq r_0$ the only solution $\eta$ that can belong to $L^2(\dd r)$ is given by
\begin{equation}
\eta(r)=\widetilde{I}(r)\Big(C_\phi-ik\alpha\beta \int_0^r\frac{g'(s)}{s^2I(s)} \phi(s)\dd s\Big),
\end{equation}
where $C_\phi$ in \cref{def:IC} is chosen so that $\eta(r)=0$ for any $r\geq r_0$. 

By the formula for $\eta$, we readily deduce that $\eta$ is compactly supported in $[0,r_0]$ since $g'(r)=0$ for $r\geq r_0$. Moreover, since $\phi \in \dot{H}^1(\dd r)$ and $\Re(\lambda_\beta)>0$ we know that $\eta \in L^2(\dd r)$, which also implies $\phi=\rmL_k^{-1}(\eta) \in \dot{H}^2(\dd r)$. We can then take derivatives of \cref{def:eta} and, upon iterating elliptic regularity on $\eta$ when needed, it is not hard to  reach the desired conclusion (note that one derivative of the factor with a $\log(r)$ contributes like dividing by $r$, thus close to the origin it behaves like differentiating once the term $r^{\alpha\beta\lambda_\beta-\frac12}$). Moreover, we can take $\alpha\beta\Re(\lambda_\beta)$ as large as needed since $\Re(\lambda_\beta)>\eps/2$ for all $\beta\geq \beta_0$.
\end{proof}
With the properties of the eigenfunction $\eta$ at hand, we are finally ready to prove Theorem \ref{th:vishik2}.
\begin{proof}[Proof of Theorem \ref{th:vishik2}]
    Let $\eta$ be the function in Corollary \ref{cor:eta}, $\phi=\rmL_k^{-1}\eta$ and $W,\Phi,F$ be defined as in \cref{eq:F}-\cref{eq:WPhi}. We have to show that 
    \begin{equation}
        \omega_{\pm}(x,t)=\frac{1}{t}(\Omega_V\pm W)\big(\frac{x}{t^\frac{1}{\alpha}},t\big)
    \end{equation}
    are equal at time $t=0$ and that for $t>0$ are two non-trivial solutions of the Euler's equation in $L^p$ with the force
    \begin{equation}
        f(x,t)=\frac{1}{t}F\big(\frac{x}{t^\frac{1}{\alpha}},t\big).
    \end{equation}
    First, we observe that since $\Omega_V(\xi)=g(|\xi|)$ is compactly supported we have 
    \begin{equation}
t^{-1}\|\Omega_V(\cdot/t^{\frac{1}{\alpha}})\|_{L^p}\lesssim t^{-1}\big (\int_\RR |g(\rho/t^{\frac{1}{\alpha}})|^p \rho \dd \rho \big)^{\frac{1}{p}}\lesssim t^{\frac{2}{\alpha p}-1}, \end{equation}
which is finite for any $p\leq 2/\alpha$ and it goes to zero as $t\to0$ for any $p<2/\alpha$. Turning our attention to the part containing $W$, we observe that
    \begin{equation}
t^{-1}\|W(\cdot/t^{\frac{1}{\alpha}})\|_{L^p}\lesssim t^{\Re(\beta \lambda)-1}\big (\int_\RR \big|\big(\frac{\eta}{\sqrt{\cdot}}\big)(\rho/t^{\frac{1}{\alpha}})\big|^p \rho \dd \rho \big)^{\frac{1}{p}}\lesssim t^{\Re(\beta\lambda)+\frac{2}{\alpha p}-1}\|r^{\frac{1}{p}-\frac12}\eta\|_{L^p}.
\end{equation}
We also know that $\lambda=\lambda_\beta+(2\alpha-1)/(2\alpha\beta)$ and therefore $\Re(\beta \lambda)=\beta \Re(\lambda_\beta)+1-1/(2\alpha)$. Moreover, when  $\beta$ is sufficiently large, thanks to the characterizations in Corollary \ref{cor:eta} we know that $\|r^{\frac{1}{p}-\frac12}\eta\|_{L^p}<\infty$. Therefore 
 \begin{equation}
t^{-1}\|W(\cdot/t^{\frac{1}{\alpha}})\|_{L^p}\lesssim t^{\beta\Re(\lambda_\beta)+\frac{1}{\alpha}(\frac{2}{ p}-\frac{1}{2})},
\end{equation}
meaning that this term goes to zero as $t\to 0$  for all $p\leq 2/\alpha$ if $\beta$ is sufficiently large (because $\Re(\lambda_\beta)>\eps/2$ for all $\beta\geq \beta_0$). By similar arguments we deduce that $f\in L^1([0,T];L^p(\RR^2))$. Hence
\begin{equation}
    \lim_{t\to 0}\|\omega_+(t)-\omega_-(t)\|_{L^p}=0
\end{equation}
for all $p\leq2/\alpha$ but $\omega_+(t)\neq \omega_-(t)$ for any $t>0$, meaning that the nonuniqueness is proved.
\end{proof}
We are then left with the proof of Theorem \ref{th:phiH1}, which is carried out in the rest of the paper. We first introduce some function spaces and operators which are needed to construct the eigenfunction in Section \ref{sec:functionspace}. Then, to construct $\phi$, we find it convenient to perform an inner-outer splitting in \cref{sec:inout}. The construction of the eigenpair is presented in several steps in Section \ref{sec:construction}.
\subsection{Function spaces and operators}
\label{sec:functionspace}
Given $ R_*>0$, we work in the standard Sobolev spaces with norms (or seminorms for $\dot{H}^1$) \begin{align}
\label{def:H10}&\|f\|_{H^1_0([0,R_*];\dd r)}^2\coloneqq \|\de_r f\|_{L^2([0,R_*];\dd r)}^2+\|r^{-1}f\|^2_{L^2([0,R_*];\dd r)},\\
\label{def:H1dot}&\|f\|_{\dot{H}^1([R_*,+\infty);\dd r)}^2\coloneqq \|\de_r f\|_{L^2([R_*,+\infty);\dd r)}^2+\|r^{-1}f\|^2_{L^2([R,+\infty);\dd r)},\\
\label{def:H1}&\|f\|_{H^1(\RR;\dd r)}^2\coloneqq \|f\|_{L^2(\RR;\dd r)}^2 +\|\de_r f\|_{L^2(\RR;\dd r)}^2+\|r^{-1}f\|^2_{L^2(\RR;\dd r)},\\
\label{def:ZM} &\|f\|_{Z_M([R_*,+\infty];\dd r)}\coloneqq M^\frac12\|f\|_{\dot{H}^1([R_*,+\infty);\dd r)}, \qquad M>0,\\
\label{def:H1r} &\|f\|_{H^1_r([0,R_*];\dd r)}:=\sum_{q=0}^1\|r^q\de_r f\|_{L^2([0,R_*];\dd r)}+\|r^{q-1}f\|_{L^2([0,R_*];\dd r)}.
\end{align}
In the sequel, we will always omit the $\dd r$ in the notation of the spaces and we simplify the notation by omitting the domain when it is clear from context in which one the norm is set. Note that the norms above distinguish between a bounded domain $[0,R_*]$ and $[R_*,+\infty)$. This is useful to exploit the compact support of the unstable vortex and handle separately technical issues that could arise from working in the whole $\RR$. Then, the first  three norms above are natural in view of the presence of the elliptic operator $\rmL_k$ in \cref{eq:rayleigh_phi}. The space $Z_M$, used for instance in \cite{kumar2023simple},  encodes smallness in $\dot{H}^1[R_*,+\infty)$ if $M$ is sufficiently large, which will allow us to treat the \textit{outer} part of the solution perturbatively. On the other hand, the space $H^1_r([0,R_*])$ is adapted to the structure of the \textit{inner} problem, where the need of controlling $r\de_r$ derivatives is related to the fact that we are working in self-similar variables.
Notice that the $H^1_r$ is the strongest norm and 
\begin{equation}
    \|f\|_{H^1_0([0,R_*])}\leq \|f\|_{H^1_r([0,R_*])} \lesssim R_*\|f\|_{H^1_0([0,R_*])}.
\end{equation}

We now introduce a key  operator arising from \cref{eq:rayleigh_phi}. First, we observe that  setting formally $\beta=\infty$, one recovers exactly the standard Rayleigh's equation where usually $\lambda=-ikc$. To construct an eigenfunction, it is  common to divide \cref{eq:rayleigh_phi} by $(ik v(r)+\lambda)$ when $\mathrm{Re}(\lambda)>0$ and start the construction around the Sturm-Liouville problem arising when $\Re(\lambda) =0$ (the neutral limiting mode procedure). Note that dividing by $ik v(r)+\lambda$ is equivalent to considering the resolvent of the transport operator $v(r)\de_{\theta}$. In our case, we need to consider the resolvent of the operator $S_\beta: D(S_\beta)\subset L^2(\dd r)\to L^2(\dd r)$ defined as
\begin{align}
    \label{def:Sbeta}S_\beta&\coloneqq-ikv(r) + \frac{1}{\alpha\beta} \big(r\partial_r(\cdot) + \frac{1}{2}\big),\\
    D(S_\beta)& =\{f\in L^2(\dd r) \, : \, r\de_r f \in L^2(\dd r) \}.
\end{align}
Moreover, $S_\beta$ is skew-adjoint on $L^2(\dd r)$. Then, we can rewrite \cref{eq:rayleigh_phi} as
\begin{equation}
\label{eq:rayleigh_phi_Sbeta}
\left(S_\beta -\lambda_\beta I\right)\rmL_k(\phi)+ik\frac{g'(r)}{r}\phi=0.
\end{equation}

\begin{remark}
\label{rem:Sbeta}
    The operator $S_\beta$ is obviously related to the operator $\mathcal{S}_\beta$ defined in \cref{eq:calSbeta}. Indeed, $S_\beta(e^{ik\theta} f)=r^\frac12\mathcal{S}_\beta r^{-\frac12}(e^{ik\theta} f)$. Namely, $S_\beta$ is the restriction to $k$-th Fourier angular mode of the operator $\mathcal{S}_\beta$ conjugated with the weight $r^\frac12$.
\end{remark}

Our next task is to provide an expansion in $\beta$ for $(S_\beta-\lambda_\beta I)^{-1}$, which is clearly well-defined when $\Re(\lambda_\beta)>0$ since $S_\beta$ is skew-adjoint. 
Once this is done, we can formalize how we treat our problem as a perturbation of the standard Rayleigh's equation. 

In the next proposition we state in which sense $S_\beta$ is a perturbation of the operator
$S_\infty: L^2(\dd r) \to L^2(\dd r)$, defined as
\begin{equation}
\label{def:Sinf}
S_\infty\coloneqq -ikv(r).
\end{equation}
\begin{proposition}
\label{prop:conv_Sbeta}
Let $S_\beta,S_\infty$ be the operators defined in \cref{def:Sbeta},\cref{def:Sinf}. For all $z\in \mathbb{C}$ satisfying $\mathrm{Re}(z)>0$, there exists $\beta_0>0$ such that for all $\beta\geq \beta_0$ the following holds true: there exists an operator $T_{z,\beta} \coloneqq T_\beta :D(S_\beta) \to L^2(\dd r)$ such that
\begin{align}
\label{eq:splitS_beta}
    (S_\beta-zI)^{-1}=-(S_\infty-zI)^{-1}-T_\beta
\end{align}
 where $-(S_\infty-zI)^{-1}=1/(ikv(r)+z)$ and
\begin{equation}
\lim_{\beta \rightarrow \infty} \|T_{\beta} \|_{X \to X} = 0, \qquad \text{ for } X\in \{L^2(\RR),\dot{H}^1(\RR),H^1_r(\RR)\}.
\end{equation}
\end{proposition}
The proof of the proposition above is postponed to the \cref{sec:proof_inverse} since it is slightly technical.

\begin{remark}
    Similarly to what we observed in Remark \ref{rem:Sbeta}, Proposition \ref{prop:conv_Sbeta} is  \cref{th:resolvent} restricted to functions concentrated on the $k$-th angular Fourier mode. To prove \cref{th:resolvent}, it is enough to replace  $r^{-1}$ with $|k|r^{-1}$ in the definition of the norms and take into account this scaling in $k$ in the proof of \cref{prop:conv_Sbeta}. These modifications are straightforward and, to avoid repetitions, we present only the details of the proof of \cref{prop:conv_Sbeta}.
\end{remark}
We now have all the main technical ingredients to start the construction.

\subsection{Inner and outer splitting}
\label{sec:inout}
To avoid some technical issues related to the unboundedness of the domain, we split our solution with an \textit{interior} and an \textit{outer} part, see also \cite{albritton2023linear,kumar2023simple}. To this end, we define two cut-off functions $\chi_1,\chi_2:[0,\infty)\to [0,1]$ with the following properties: let $R=10r_0$ with $r_0$ the fixed constant  in the \cref{def:class} and let $M\gg r_0$ to be specified later.
Let $\chi_1$ be a smooth function such that $\chi_1(r)=1$ for $r\leq M/2$ and $\chi_1(r)=0$ for $r\geq M$. This is the cut-off for the \textit{interior} part. For the \textit{outer} part, we set $\chi_2(r)=0$ for $r\leq R$ and $\chi_2(r)=1$ for $r\geq 2R$. Moreover, these cut-off can be chosen to satisfy
\begin{align}
&\|\chi_1'\|_{L^\infty}+\|r\chi_1''\|_{L^\infty}\lesssim \frac{1}{M},\label{eq:boundschi1}\\
&\|\chi_2'\|_{L^\infty}+\|r\chi_2''\|_{L^\infty}\lesssim 1,\label{eq:boundschi2}
\end{align}
where we keep track of constants depending on $M$ since this is a parameter that needs to be choosen sufficiently large. On the other hand, all constants depending on $R$ will be hidden in the $\lesssim$.
Then, we split 
\begin{equation}
      \phi= \chi_1\phi_{\rmin}+\chi_2\phi_{\rmout},
\end{equation}
and we need to  define the equations satisfied by $\phi_{\rmin}$ and $\phi_{\rmout}$. We observe that 
\begin{align}
\label{com:chiL}
[\chi_j,S_\beta \rmL_k](\phi)&=\left(-ikv(r)+\frac{1}{2\alpha\beta}\right)[\chi_j,\rmL_k](\phi)+\frac{1}{\alpha\beta}[\chi_j,r\de_rL_k](\phi),
\end{align} 
where  the commutators between $\chi_j$ and the operators involved are
\begin{align}
[\chi_j,\rmL_k](\phi)& =[\chi_j,\de_{rr}](\phi)=-2\chi_j'\phi'-\chi_j''\phi, \\
    \label{com:chirL}[\chi_j,r\de_r\rmL_k](\phi)& =[\chi_j,r\de_{rrr}](\phi)  - [\chi_j,r\de_{r} (r^{-2}(k^2-1/4))](\phi) \nonumber\\
    & =-r\de_r(\chi_j''\phi+2\chi_j'\phi')+ (k^2-1/4) r^{-1} \chi_j'\phi.
\end{align}
We finally define the inner and outer problems as 
\begin{align}
\label{eq:inner}&\left(S_\beta -\lambda_\beta I\right)\rmL_k(\phi_{\rmin})+ik\frac{g'(r)}{r}\phi_{\rmin}=\mathcal{C}_2(\phi_{\rmout}), \qquad  &\text{ for } r\in [0,M]\\
\label{eq:outer}&\left(S_\beta -\lambda_\beta I\right)\rmL_k(\phi_{\rmout})=\mathcal{C}_1(\phi_{\rmin}), \qquad& \text{ for } r\in [R,+\infty)
\end{align}
both with Dirichlet boundary conditions, where 
\begin{align}
   \label{def:Rout} &\mathcal{C}_2(\phi_{\rmout})\coloneqq[\chi_2,(S_\beta  - \lambda_\beta I) \rmL_k](\phi_{\rmout})\\
        \label{def:Rin}&\mathcal{C}_1(\phi_{\rmin})\coloneqq[\chi_1,(S_\beta  - \lambda_\beta I) \rmL_k](\phi_{\rmin}).
\end{align}
Since $\chi_2g'=0$, $\chi_1=1$ when $\chi_2'\neq 0$ and vice versa, it is not hard to see that when $\phi_{\rmin},\phi_{\rmout}$ solve \cref{eq:inner} and \cref{eq:outer} respectively,  then $\phi$ solves \cref{eq:rayleigh_phi_Sbeta}.

\subsection{Self-Similar Rayleigh as perturbation of Rayleigh}
Thanks to Proposition \ref{prop:conv_Sbeta}, we can rewrite our inner and outer problems \cref{eq:inner}-\cref{eq:outer} as follows:  in analogy with the common notation in the literature, we introduce the constant
\begin{equation}
\label{def:c}
c\coloneqq-\frac{\lambda_\beta}{ik}    
\end{equation}
Then, we consider $T_\beta$ as the perturbation we were looking for and, thanks to Proposition \ref{prop:conv_Sbeta} with $z=-ik c$, we rewrite the inner problem as
\begin{align}
\label{eq:rayleigh_phi_Sbeta_final_inner}
\rmL_k(\phi_{\rmin})-\frac{g'(r)}{r(v(r)-c)}\phi_{\rmin} =T_\beta\left(ik\frac{g'(r)}{r}\phi_{\rmin}\right)+\left(S_\beta -\lambda_\beta I\right)^{-1}\left(\mathcal{C}_2(\phi_{\rmout})\right).
\end{align}
For the outer problem we do not need a detailed expansion and therefore we simply have 
\begin{align}
\label{eq:Lkout}
\phi_{\rmout}=\left(\rmL_k^{-1}\left(S_\beta -\lambda_\beta I\right)^{-1}\right)\left(\mathcal{C}_1(\phi_{\rmin})\right)=:\cR_1(\phi_{\rmin}).
\end{align}
\subsection{Constructing the eigenpair}
\label{sec:construction}
To prove the existence of the desired eigenpair, we now follow more closely the strategy proposed in \cite{kumar2023simple} (which is itself based on \cite{vishik_notes,albritton2023linear,lin2003instability}), which we divide in 5 steps below. The first key step is to extract a Sturm-Liouville (SL) eigenvalue problem  approximating \cref{eq:rayleigh_phi_Sbeta_final_inner} in a limiting regime. In our case, this can be done by formally setting $\beta=\infty, \phi_{\mathrm{out}}=0$, $k=k_0$ and $c=v(1)$. Note that the latter choice avoids potentially dangerous singularities in view of the specific properties of a vortex in the class $\mathfrak{G}$ (namely, $g'(r)/(v(r)-v(1))$ is bounded for $r\to1$). Then, as we explain in Step 1, one can find an eigenpair $(\lambda_0,\phi_0)$ for the SL problem with $\lambda_0$ related to $k_0$ and $\phi_0$ will in fact be bounded in $H^1_r$. This is the starting point of the perturbative construction with suitable parameters at play. Indeed, in Step 2, we consider  $\phi_{\mathrm{in}}=\phi_0+\varphi$ with $\varphi\perp \phi_0$ in $L^2$, the parameter $0<\eps\ll1$, we set 
\begin{equation}
\label{def:k}
\mathbb{N} \ni k^2=k_0^2+\eps.
\end{equation}
and we look for $\tilde{c}$ so that
\begin{equation}
\label{def:c_with_eps}
c=v(1)+\eps \tilde{c}.
\end{equation}
We would then need to solve a system for $\varphi,\phi_{\mathrm{out}}$ and $\tilde{c}$ is determined via the orthogonality condition $\varphi \perp \phi_0$. The construction of $\varphi,\phi_{\mathrm{out}}$  can be reduced in checking the invertibility of certain operators in the spaces introduced at the beginning of \cref{sec:functionspace}. This will be carried out in great detail in Steps 2-4 and represents the most technical part of the paper. Once this is done, it remains to find $\tilde{c}$ and prove that $\Im(\tilde{c})>0$, a task carried out in Step 5. The existence is essentially a consequence of Rouché's theorem whereas $\Im(\tilde{c})>0$ is a consequence of the Sochocki-Plemelj formula.
We finally remark that the main idea can be summarized as an instability arising from a \textit{neutral limiting mode} associated to the problem with $\eps=0$. 

\medskip

\textbf{Step 1:} We define the approximating problem for  $\phi_{\rmin}$ by considering the formal limits $\eps\to 0$ and $\beta\to \infty$, which give a regular Sturm-Liouville eigenvalue problem
\begin{align}
\label{eq:phi_0}
&-\rmL_{k_0}(\phi_0) + A(r) \phi_0 = -\phi_0''+A(r)\phi_0+\frac{1}{r^2}(k_0^2-\frac14)\phi_0=0, \qquad \text{for } r\in [0,M]\\
\label{def:A}&A(r)\coloneqq\frac{g'(r)}{r(v(r)-v(1))},
\end{align}
with Dirichlet boundary conditions $\phi_0(0)=\phi_0(M)=0$. The standard form of a Sturm-Liouville problem, following the notation in \cite{Kong96}, is $-(py')'+qy=\lambda wy$. In our case, we  identify $y=\phi_0$, $p=1$, $q=A(r)$, $w=r^{-2}$ and $\lambda=-(k_0^2-1/4)$. In the sequel we will apply results from \cite{Kong96}, which only requires that $1/p,q,w\in L^1_{loc}((0,M))$ which are clearly satisfied in our context. Now, the parameter $k_0$ can be chosen so that $\lambda$ is the lowest eigenvalue for \cref{eq:phi_0} (and $\phi_0$ the associated eigenfunction). Thanks to the boundary conditions and $p\geq0$, we know that all eigenvalues are simple and eigenfunctions form a basis for $L^2([0,M])$ \cite{Kong96} Thus, the lowest eigenvalue is characterized by the Rayleigh's quotient as 
\begin{equation}
\label{eq:Rayquot}
    -(k_0^2-\frac14)=\min_{\|f/r\|_{L^2}=1} \int_0^{M}|f'(r)|^2+A(r)|f(r)|^2 \dd r.
\end{equation}
For vortices as in \cref{def:class}, for  $r\to 1$ we have $A(r)= \gamma_1/v'(1)+\smallO(r-1)$, where we recall that $v'(1)=-2\int_0^1g(r)r\dd r+g(1)=-2\int_0^1g(r)r\dd r-g_1$, where the last identity follows by property iii). Choosing appropriately the values of $g_0,g_1,\gamma_1, \delta_0,\delta_1$, one can choose $A(1)\approx -\widetilde{\gamma}_1<0$ for a given $\widetilde{\gamma}_1$,  and therefore $k_0^2-1/4$ can be tuned by choosing appropriately $\widetilde{\gamma}_1$ (one can test the minimization problem \cref{eq:Rayquot} with a function concentrated close to $r=1$). Moreover, by standard properties of regular Sturm-Liouville problems \cite{Kong96}, we know that the eigenvalue $k_0^2-1/4$ depends smoothly on the choices of the parameters in the class $\mathfrak{G}$ in \cref{def:class}. Therefore, we  know that for any $\eps$ there exists a  vortex in the class $\mathfrak{G}$ such that the conditions \cref{def:k} and \cref{eq:Rayquot} are satisfied. In fact, to prove that there exists a vortex in the class $\mathfrak{G}$ so that $k^2$ in \cref{def:k} is a natural number, one can also choose an arbitrary $|k_0|\geq 2$ and then perform a suitable linear combination of vortices in the class $\mathfrak{G}$, as done in Proposition 4.4.6 and Lemma 4.4.7 of \cite{vishik_notes}.

Let us prove that $\phi_0\in H^1_r([0,M])$ with a bound on this norm that does not depend on $M$ (otherwise it is straighforward since  by \cref{eq:Rayquot} we know $\phi_0\in H^1_0([0,M])$ ). Indeed, we can also prove better weighted estimates if $|k_0|$ is sufficiently large. To do this, we test \cref{eq:phi_0} with $-r^{2q}\phi_0$ for $q\geq 0$ to  obtain that 
\begin{equation}
    -\int_0^M r^{2q}\phi_0''\phi_0\dd r+\big(k_0^2-\frac14\big)\|r^{q-1}\phi_0\|^2_{L^2}=-\int_0^M r^{2q}A(r)\phi_0^2\dd r\leq \|r^{2(q+1)}A\|_{L^\infty}\|r^{-1}\phi_0\|_{L^2}.
\end{equation}
Integrating by parts we get
\begin{equation}
  -\int_0^M r^{2q}\phi_0''\phi_0\dd r=\|r^q\phi_0'\|_{L^2}^2+2q\int_0^Mr^{2q-1}\phi_0'\phi_0 \dd r.
\end{equation}
Hence, since $\|\phi_0/r\|_{L^2}=1$,
\begin{equation}
\|r^{q}\phi_0'\|_{L^2}^2+\big(k_0^2-\frac14\big)\|r^{q-1}\phi_0\|^2_{L^2}\leq \|r^{2(q+1)}A\|_{L^\infty}+2q\|r^q\phi_0'\|_{L^2}\|r^{q-1}\phi_0\|_{L^2}.
\end{equation}
Since $g'$ is supported in $[0,r_0]$, we see that $|r^{2(q+1)}A(r)|\leq C_1$ for some constant that does not depend on $M$. Hence, by Young's inequality
\begin{equation}
\label{eq:bdrqphi0}
\|r^{q}\phi_0'\|_{L^2}^2+\big(k_0^2-\frac14\big)\|r^{q-1}\phi_0\|^2_{L^2}\leq C_1+\frac{q}{c}\|r^q\phi_0'\|_{L^2}^2+c\|r^{q-1}\phi_0\|_{L^2}^2
\end{equation}
for any constant $c>0$. Choosing $q=1$ and $c=3/2$ we see that for any $|k_0|\geq 2$ we can easily absorb the last term in the left-hand side and get 
\begin{equation}
    \label{bd:phi0L2}
\|r\phi_0'\|_{L^2}+\|\phi_0\|_{L^2}\lesssim 1
\end{equation}
with an hidden constant that does not depend on $M$. Hence $\phi_0\in H^1_r([0,M])$ with a norm bound independent of $M$. Moreover, setting $q=3$ and $c=3$ in \eqref{eq:bdrqphi0} we also get that for $|k_0|\geq 2$
\begin{equation}
    \label{bd:rqphi0}
    \|r^2\phi_0\|_{L^2}\lesssim 1,
\end{equation}
which will be useful later.
\medskip 

\textbf{Step 2:} 
We now look for a solution to the inner problem in the form 
\begin{equation}
\label{eq:ansatz_phi_min}
\phi_{\rmin}=\phi_0+\varphi.
\end{equation}
We also need to impose that $\varphi\perp \phi_0$. Thus, we require  
\begin{equation}
\label{eq:proj} 
\mathbb{P}(\varphi)\coloneqq\frac{1}{\|\phi_0\|_{L^2}\|\varphi\|_{L^2}}\langle{\varphi,\phi_0}\rangle\phi_0=0,
\end{equation}
where the inner product above is in $L^2(\dd r)$ and the scaling constant is chosen for later convenience.
We can  insert the ansatz \cref{eq:ansatz_phi_min} in \cref{eq:rayleigh_phi_Sbeta_final_inner} and 
impose the last condition \cref{eq:proj}. In particular, given that $\phi_0$ solves \cref{eq:phi_0}, we can rewrite the left-hand side of \cref{eq:rayleigh_phi_Sbeta_final_inner} as 
\begin{equation}
    \rmL_{k_0}(\varphi)-(A(r)+\mathbb{P})(\varphi)-\eps\left(\frac{\tilde{c} g'(r)}{r(v(r)-c)(v(r)-v(1))}+\frac{1}{r^2}\right)(\phi_0+\varphi).
\end{equation}
Note that the factor $1/r^2$ above arise from \cref{def:k}, since $\rmL_{k}=\rmL_{k_0}-\eps/r^2$.
Applying $\rmL_{k_0}^{-1}$, the resulting system for $\varphi$ can then be rewritten as follows
\begin{equation}
   \label{eq:varphi}
    \begin{cases}
        (I-\cK)(\varphi)=(\cR_{\eps}+\cR_\beta)(\phi_0+\varphi)+\cR_{2}(\phi_{\rmout}),\\
        \mathbb{P}(\varphi)=0,
    \end{cases}
\end{equation}
where we define 
\begin{align}
    \label{def:K} &\cK(f)\coloneqq\rmL_{k_0}^{-1}((A(r)+\mathbb{P})(f))\\
   \label{def:cReps} &\cR_{\eps}(f)\coloneqq\eps \, \rmL_{k_0}^{-1}\left(\frac{g'(r)\tilde c f}{r(v(r)-c)(v(r)-v(1))}+\frac{f}{r^2}\right)\\
   \label{def:cRbeta} &\cR_{\beta}(f)\coloneqq\rmL_{k_0}^{-1}\big( T_\beta \big(ik  \frac{g'(r)}{r} f\big)\big)\\
&\cR_{2}(\phi_{\rmout})\coloneqq\rmL_{k_0}^{-1}\big((S_\beta-\lambda_\beta I)^{-1}\mathcal{C}_2(\phi_{\rmout})\big)
\end{align}
We want to solve  the first equation in \cref{eq:varphi} in $H^1_r([0,M])$ with Dirichlet boundary conditions.
 We observe the following. 
 \begin{lemma}
 \label{lemma:I-K}
     The operator $I-\cK:H^{1}_r([0,M])\to H^1_r([0,M])$ is invertible. Moreover, we have 
     \begin{equation}
     \|(I-\cK)^{-1}f\|_{H^1_r([0,M])}\leq C\|f\|_{H^1_r([0,M])}     \end{equation} with $C$ independent of $M$.
 \end{lemma}
 \begin{proof}
     The proof of the invertibility in $H^1_0([0,M])$ is similar to \cite[Lemma 3]{kumar2023simple}. Since $\cK$ is compact we only have to show that $\mathrm{Ker}(I-\cK)=0$, namely that if 
     \begin{equation}
         \rmL_{k_0}\psi-(A+\mathbb{P})\psi=0
     \end{equation}
     then $\psi=0$. Testing against $\phi_0$, since $\rmL_{k_0}-A$ is self-adjoint and $(\rmL_{k_0}-A)(\phi_0)=0$ we deduce that $\mathbb{P}\psi\perp \phi_0$, which is only possible if $\psi=0$. Moreover, the bound on the  $H^1_0([0,M])$ norm of $(I-\cK)^{-1}$ only depends on the spectral gap associated to the Sturm-Liouville problem \cref{eq:phi_0}. Indeed, if $\psi=(I-\cK)^{-1}f\in \mathrm{span}\{\phi_0\}$ we have nothing to prove thanks to the properties of $\phi_0$. Otherwise, since $k_0^2-1/4$ is the maximal eigenvalue of $-\de_{rr}+A(r)$, by standard Sturm-Lioviulle theory \cite{Kong96} we have a spectral gap that only depends on the properties of the function $A$, which itself does not depend on $M$.
     
     We thus have to show the bound only on the part with $q=1$ in the $H^1_r$ norm. By the definition of the operators,  let $\psi=(I-\cK)^{-1}f$. Then
\begin{equation}
         \rmL_{k_0}\psi-(A(r)+\mathbb{P})(\psi)=\rmL_{k_0}f.
     \end{equation}
 Testing the equation by $-r^{2}\psi$, as done to obtain \cref{bd:phi0L2}, we first observe that for $|k_0|\geq 2$
 \begin{equation}
     -\langle r^{2}\rmL_{k_0}\psi,\psi\rangle \approx \|r\psi\|_{L^2}^2+\|\psi\|_{L^2}^2.
 \end{equation}
 Thus, by Cauchy-Schwarz and the definition of $\mathbb{P}$, it is not hard to deduce that 
 \begin{align}
\|r\psi'\|_{L^2}^2+\|\psi\|_{L^2}^2\lesssim \, &\|r^4A\|_{L^\infty}\|r^{-1}\psi\|_{L^2}^2+\|r^2\phi_0\|_{L^2}\|\psi\|_{L^2}+\|f\|_{H^1_r}(\|r\psi'\|_{L^2}+\|\psi\|_{L^2}).
 \end{align}
 We conclude the proof thanks to the following properties: i)  the piece with $r^4A$ is uniformly bounded thanks to the properties of $g$. ii) the part involving $\phi_0$ is bounded thanks to the bounds \cref{bd:phi0L2,bd:rqphi0} on $\phi_0$. iii) the term $r^{-1}\psi$ is controlled with the $H^1_0$ norm uniformly in $M$ thanks to the observation we made before. 
 \end{proof}
 \medskip 

 \textbf{Step 3:} Now we rewrite the inner and outer problems as a perturbation of the identity operator, where for the inner problem \cref{eq:rayleigh_phi_Sbeta} we use the formulation presented in \cref{eq:varphi}. Applying $(I-\cK)^{-1}$ on the left-hand side of the first equation of \cref{eq:varphi} and recalling \cref{eq:Lkout}, we see that 
 \begin{align}
 \label{eq:innerouter}
  &\left(I_{2\times 2}-\begin{pmatrix}
    \cM_{1,1} &\cM_{1,2}\\
      \cM_{2,1} & 0
  \end{pmatrix}\right) \begin{pmatrix}
      \varphi \\
      \phi_{\rmout}
  \end{pmatrix}=\begin{pmatrix}
      \cM_{1,1}(\phi_0)\\
      \cM_{2,1}(\phi_0)
  \end{pmatrix}\\
  \label{eq:projection}&\mathbb{P}(\varphi)=0.
 \end{align}
where we define the operators 
\begin{align}
   \label{def:M11}&\cM_{1,1}\coloneqq(I-\cK)^{-1}(\cR_\eps+\cR_{\beta}),\\
   &\cM_{1,2}\coloneqq(I-\cK)^{-1}\cR_2,\\
   &\cM_{2,1}\coloneqq\cR_1=   \rmL_{k_0}^{-1}(S_\beta-\lambda_\beta I)^{-1}\mathcal{C}_1.
\end{align}
To solve the first problem in \cref{eq:innerouter}, we need to guarantee that the operator norms of $\cM_{i,j}$ are small. Indeed, denoting the matrix $\cM=(\cM_{i,j})_{i,j=1}^2$ with $\cM_{2,2}=0$, we see that if the Neumann series of $(I-\cM)^{-1}$ makes sense then 
\begin{equation}
\label{eq:solvaprhi}
    \begin{pmatrix}
      \varphi \\
      \phi_{\rmout}
  \end{pmatrix}=(I_{2\times 2}-\cM)^{-1}\begin{pmatrix}
      \cM_{1,1}(\phi_0)\\
      \cM_{2,1}(\phi_0)
  \end{pmatrix}=\sum_{l=0}^\infty \cM^l\begin{pmatrix}
      \cM_{1,1}(\phi_0)\\
      \cM_{2,1}(\phi_0)
  \end{pmatrix}.
\end{equation}

\medskip

\textbf{Step 4:} We collect the bounds for the operators $\cM_{i,j}$ in the next proposition, which is at the core of the argument. 
\begin{proposition}
\label{prop:boundsM}
If $\beta\gg \eps^{-2}$, the following bounds hold:
\begin{align}
\label{bd:M11}\| \cM_{1,1}\|_{H^1_r([0,M]) \rightarrow H^1_r([0,M])} & = \mathcal O(\varepsilon M^4)+M^2\smallO_\beta(1), \\
\| \cM_{1,2}\|_{Z_M([R, +\infty)) \rightarrow H^1_r([0,M])}& = \cO(M^{-1/2}),\\
\|\mathcal \cM_{2,1}\|_{H^1_r
([0,M]) \rightarrow Z_M([R,+\infty))} & = \cO(M^{-1/2}),
\end{align}
where $\smallO_\beta(1)$ is a function that goes to zero as $\beta\to \infty$.
\end{proposition}
With this proposition at hand, we see that \cref{eq:solvaprhi} is indeed well-defined.
To prove \cref{prop:boundsM}, we first recall the following lemma from \cite[Lemma 4]{kumar2023simple}:
\begin{lemma}
\label{lemma:approximation_denominator}
For $\eps \rightarrow 0$,
\begin{equation}
h_\eps(r) \coloneqq \frac{1}{v(r)-c}- \frac{1}{v'(1)(r-1)-\eps \tilde{c}} = \mathcal{O}(1) + i \mathcal{O}(\eps),
\end{equation}
uniformly for $r>0$.
\end{lemma}

\begin{proof}
In the proof, we will use the same letter $\psi$ to denote the solution to different auxiliary problems we are going to introduce.

\medskip
\noindent $\diamond$
\textbf{Bound for $\mathbf{\cM_{1,1}}$}: We first prove a bound on the $H^1_0$ norm of the operator. We can observe that
\begin{equation}
\label{eq:M_decomposition}
   \| \cM_{1,1}\|_{H^1_0\rightarrow H^1_0} \leq  \| (I-\mathcal K)^{-1}\|_{H^1_0\to H^1_0} \| \mathcal R_\varepsilon + \mathcal R_\beta \|_{H^1_0\to H^1_0}.
\end{equation}
Since by \cref{lemma:I-K}, $(I-\mathcal K)^{-1}$ is a bounded operator in $H^1_0$, we are left to analyze $\mathcal R_\varepsilon$ and $\mathcal R_\beta$. From \cref{lemma:approximation_denominator}, we can conclude that there exists a bounded function $h_\eps$ with $|h_\eps| = \mathcal{O}(1)$ uniformly in $r$,  such that $\mathcal R_\varepsilon$ can be rewritten as 
\begin{equation}
    \mathcal R_\varepsilon (f) = \varepsilon \rmL_{k_0}^{-1} \left(A(r) \tilde{c} \left(\frac{f}{v'(1)(r-1)- i \eps  \Im \tilde{c}} + h_\eps(r)f\right)+ \frac{f}{r^2}\right).
\end{equation}
To bound the operator norm, let us consider an arbitrary $f \in H^1_0$ and let $\psi = \mathcal R_\varepsilon f$, namely,
\begin{equation}
\label{eq:reps_eqn}
\partial_{rr}\psi - \frac{1}{r^2} \left(k_0^2-\frac{1}{4}\right) \psi = A(r) \eps \tilde{c} \left(\frac{1}{v'(1)(r-1)-i \eps \Im \tilde{c}} + h_\eps(r) \right) f + \varepsilon \frac{f}{r^2}.
\end{equation}
Note that we can rewrite the first term in the right-hand side as
\begin{align}
\label{eq:log_rewriting}
    \frac{1}{v'(1)(r-1)-i \eps \Im \tilde{c}} = \frac{1}{v'(1)} \de_r \left(\log\sqrt{v'(1)^2(r-1)^2 + (\varepsilon \Im\tilde{c})^2} + \taninv\frac{\varepsilon \Im \tilde{c}}{v'(1)(r-1)}\right),
\end{align}
and in particular there exists a constant $C>0$ independent of $\eps$ such that 
\begin{equation}
\|\log\sqrt{v'(1)^2(\cdot-1)^2 + (\varepsilon \Im\tilde{c})^2} \|_{L^4([0,M])}\leq C M.
\end{equation}
Recalling also that $v'(1) \neq 0$ and $A(r),A'(r)$ are bounded, by integration by parts and H\"older's inequality we get 
\begin{align}
    \bigg|\bigg\langle &\frac{\varepsilon A(r) \tilde c f}{v'(1)(r-1) - i \varepsilon \Im \tilde c}, \psi \bigg\rangle\bigg| \lesssim \varepsilon\left|\left\langle \frac{A(r)}{v'(1)}  f  \de_r \left(\log\sqrt{v'(1)^2(r-1)^2 + (\varepsilon \Im\tilde{c})^2} + \taninv\frac{\varepsilon \Im \tilde{c}}{v'(1)(r-1)}\right), \psi \right\rangle\right| \\
    & \lesssim \varepsilon \left\langle |f| \left|\log\sqrt{v'(1)^2(r-1)^2 + (\varepsilon \Im\tilde{c})^2} + \taninv\frac{\varepsilon \Im \tilde{c}}{v'(1)(r-1)}\right|, \de_r \psi \right\rangle \\ 
    & \quad + \varepsilon\left\langle |\de_r f| \left|\log\sqrt{v'(1)^2(r-1)^2 + (\varepsilon \Im\tilde{c})^2} + \taninv\frac{\varepsilon \Im \tilde{c}}{v'(1)(r-1)}\right|, |\psi| \right\rangle\\
    &\quad +\varepsilon\left\langle |f| \left|\log\sqrt{v'(1)^2(r-1)^2 + (\varepsilon \Im\tilde{c})^2} + \taninv\frac{\varepsilon \Im \tilde{c}}{v'(1)(r-1)}\right|, |\psi| \right\rangle\\
    & \lesssim \varepsilon (M\|f'\|_{L^2}\|\psi\|_{L^4} + M\|f\|_{L^4}(\|\psi'\|_{L^2}+\|\psi\|_{L^2}) + \|f'\|_{L^2}\|\psi\|_{L^2} + \|f\|_{L^2}\|\psi'\|_{L^2}).
\end{align}
Therefore, testing \cref{eq:reps_eqn} with $\psi$  it is not hard to conclude that 
\begin{align}
    \| \psi\|^2_{H^1_0} 
     & \lesssim \varepsilon (M\|f'\|_{L^2}\|\psi\|_{L^4} + M\|f\|_{L^4}\|\psi'\|_{L^2} + \|f'\|_{L^2}\|\psi\|_{L^2}  \\ & \qquad +\|f\|_{L^2}\|\psi'\|_{L^2}+ \| r^{-1} f\|_{L^2} \| r^{-1} \psi\|_{L^2} +  M \|f\|_{L^2}\|\psi\|_{L^2}) .
\end{align}
Then, we observe that for any $f\in H^1_0([0,M])$ by Gagliardo-Nirenberg inequality we have
\begin{align}
    &\| f\|_{L^4([0,M])}\lesssim  \|f'\|_{L^2([0,M])}^\frac14\|f\|_{L^2([0,M])}^\frac34\lesssim M^\frac34 \|f\|_{H^1_0([0,M])} \label{eq:L4M}\\
    &\| f\|_{L^2([0,M])} \lesssim M\|f\|_{H^1_0([0,M])}\label{eq:L2M}.
\end{align}
Hence, we can conclude that
\begin{equation}
\|\psi\|_{H^1_0} = \| \mathcal R_{\varepsilon} f \|_{H^1_0} \lesssim \eps M^2 \|f\|_{H^1_0}\quad \forall f \in H^1_0([0,M]), 
\end{equation}
meaning that 
\begin{equation}
\label{eq:Reps_est}
\| \mathcal R_{\varepsilon} \|_{H^1_0([0,M]) \rightarrow H^1_0([0,M])} = \mathcal{O}(\varepsilon M^3). 
\end{equation}

For what concerns $\mathcal R_\beta$, we set $\psi = \rmL_{k_0}^{-1}\left( T_\beta \left(ik r^{-1} g'(r)f \right)\right)$, so that
\begin{align}
    \partial_{rr} \psi - \frac{1}{r^2} \left(k_0^2-\frac{1}{4}\right) \psi = T_\beta \left(ik \frac{g'(r)}{r}f\right).
\end{align}
By testing with $\psi$, exploiting the properties of $g$ in the class $\mathfrak{G}$, we get the bound
\begin{equation}
\| \psi \|^2_{H^1_0} \lesssim \|T_\beta\|_{L^2\rightarrow L^2} \|f\|_{H^1_0} \|\psi\|_{L^2}\lesssim M\|T_\beta\|_{L^2\rightarrow L^2} \|f\|_{H^1_0} \|\psi\|_{H^1_0}.
\end{equation}
Thanks to the bound on $T_\beta$ in \cref{prop:conv_Sbeta}, we know that
\begin{equation}
\label{eq:Rbeta_est}
 \|\mathcal R_\beta \|_{H^1_0([0,M])\to H^1_0([0,M])} =M\smallO_\beta(1)
\end{equation}
Finally, since for $\|f\|_{H^1_r([0,M])}\lesssim M\|f\|_{H^1_0([0,M])}$, exploiting \cref{eq:Reps_est} and \cref{eq:Rbeta_est} in \cref{eq:M_decomposition}, we prove \cref{bd:M11}.

\medskip

\noindent $\diamond$ \textbf{Bound for $\mathbf{\cM_{1,2}}$}: By \cref{lemma:I-K}, we know that 
\begin{equation}
    \|(I-\cK)^{-1}\cR_2(f)\|_{H^1_r}\lesssim \| \cR_2(f)\|_{H^1_r}.
\end{equation}
To control the norm involved above, we set $\psi = \mathcal R_2(f)$ so that $(S_\beta-\lambda_\beta I)\rmL_{k_0} \psi = 
\mathcal{C}_2(f)$. Then, observe that for $q\in \{0,1\}$ we have 
\begin{equation}
    [S_\beta-\lambda_\beta I,r^{2q}]=-\frac{2q}{\alpha \beta}r^{2q}
\end{equation}
meaning that
\begin{equation}
    (S_\beta-\lambda_\beta I)(r^{2q}\rmL_{k_0} \psi )= -\frac{2q}{\alpha \beta}r^{2q}\rmL_{k_0}\psi+r^{2q}
\mathcal{C}_2(f)
\end{equation}
and consequently
\begin{equation}
\label{eq:rqpsi}
   r^{2q}\rmL_{k_0} \psi=-\frac{2q}{\alpha\beta}(S_\beta-\lambda_\beta I)^{-1}(r^{2q}\rmL_{k_0}\psi) +(S_\beta-\lambda_\beta I)^{-1}(r^{2q}\mathcal{C}_2(f)).
\end{equation}
Now, to control the $H^1_r$ we want to test with $-\psi$ and sum in $q$. This is because for $|k_0|\geq 2$
\begin{equation}
- \sum_{q=0}^1\langle r^{2q}\rmL_{k_0} \psi,\psi\rangle \approx \|\psi\|_{H^1_r}^2.
\end{equation}
But first, we observe that 
since $S_\beta$ is skew-adjoint, we have  $(S_\beta-\lambda_\beta I)^* = (-S_\beta -\bar{\lambda}_\beta I)$. On the operator $(-S_\beta - \bar{\lambda}_\beta I)^{-1}$ we can draw the same conclusions we got for $(S_\beta - \lambda_\beta I)^{-1}$ in \cref{prop:conv_Sbeta}, because the estimates do not depend on the sign of the \textit{imaginary} part of $\lambda_\beta$ (i.e., the \textit{real} part of $c$) and the operator $-S_\beta$ is skew-adjoint. Hence, there exists an operator $\tilde T_{\beta}\coloneqq \tilde T_{z,\beta} :D(S_\beta) \to L^2(\dd r)$ such that:
\begin{equation}
    (S_\beta-\bar{\lambda}_\beta I)^{-1}=-\frac{1}{ik(v(r)-\bar c)}-\tilde T_\beta, \quad \text{with}\quad 
\lim_{\beta \rightarrow \infty} \|\tilde T_{\beta} \|_{H^1_r \to H^1_r} = 0.
\end{equation}
Then, testing with $-\psi$ and summing for $q=0,1$,  for $|k_0|\geq 2$ we deduce that
\begin{align}
\notag
\|\psi\|^2_{H^1_r}
   \lesssim  \frac{1}{\beta}|\langle r^{2}\rmL_{k_0}\psi, (-S_\beta-\bar{\lambda}_\beta I)^{-1} \psi \rangle | 
+\sum_{q=0}^1|\langle r^{2q}\mathcal{C}_2(f), (-S_\beta-\bar{\lambda}_\beta I)^{-1} \psi \rangle|.
\end{align}
Observe that thanks to the properties of $T_\beta$ (and hence $\tilde{T}_\beta$) and the fact that $\Re(\lambda_\beta)\gtrsim \eps$, we know that 
\begin{equation}
   \sum_{q=0}^1 \|r^q\de_r(-S_\beta-\bar{\lambda}_\beta I)^{-1} \psi\|_{L^2}+\|r^{q-1}(-S_\beta-\bar{\lambda}_\beta I)^{-1} \psi\|_{L^2}\lesssim \frac{1}{\eps^2}\|\psi\|_{H^1_r}
\end{equation}
Notice that the loss in the last inequality is related to the singularity of $r\de_r(ikv(r)-\lambda_\beta)^{-1}=-ikrv'(r)/(ikv(r)-\lambda_\beta)^2$, which we are estimating in a brutal with $\Re(\lambda_\beta)\gtrsim \eps$ and $|rv'|\lesssim 1$. This is however sufficient since we have an extra factor of $\beta$ to compensate the loss. 
Therefore,  integrating by parts  and using the Cauchy-Schwarz inequality we find that 
\begin{equation}
   \frac{1}{\beta} |\langle r^{2}\rmL_{k_0}\psi, (-S_\beta-\bar{\lambda}_\beta I)^{-1} \psi \rangle |\lesssim \frac{1}{\eps^2\beta}\|\psi\|_{H^1_r}^2
\end{equation}
which can be absorbed in the left-hand side whenever $\beta_0\gg \eps^{-2}$. We are thus left with the bound involving $\mathcal{C}_2$. Observe that we can rewrite $\mathcal C_2(f)$ as
\begin{align}
\label{eq:c2rewritten}
\mathcal{C}_2(f) & =  [\chi_2, (S_\beta - \lambda_\beta I)L_{k_0}](f) \\ & = \left(-ikv(r) + \frac{1}{2\alpha \beta}\right) (-2\chi'_2 f' - \chi''_2 f) + \frac{1}{\alpha \beta} r\de_r (-2\chi'_2 f' - \chi''_2 f) \\ & \qquad + \frac{1}{\alpha \beta} (k_0^2 -1/4) r^{-1} \chi'_2 f - \left(\lambda - \frac{\alpha -1}{\alpha\beta}\right) (-2\chi'_2 f - \chi''_2 f).
\end{align}
Thus, this term is supported in $[R,2R]$ where $(ikv(r)-c)^{-1}$ does not have singularities and all the factors of $r$ are uniformly bounded, meaning that we can use the $\dot{H}^1$ norm on $f$ without loosing anything. Indeed, we see that when we bound $|\langle r^{2q}\mathcal{C}_2(f), (-S_\beta-\bar\lambda_\beta I)^{-1} \psi \rangle|$, the terms 
with $\chi'_2 f'$ and $\chi''_2 f$ show the same behavior in the bounds, because by \cref{eq:boundschi2} we have
\begin{equation}
\|r^{2q+1}\chi''_2  (r^{-1}f)\|_{L^2} \lesssim \|f\|_{\HH}, \quad \|\chi'_2 f'\|_{L^2} \lesssim \|f\|_{\HH} 
\end{equation}
We will therefore only show the bounds for the terms with $\chi'_2 f'$. Defining $\Psi = (-S_\beta-\bar\lambda_\beta I)^{-1} \psi$, we can start bounding the terms with a factor $\beta^{-1}$. Integrating by parts, exploiting the support of $\chi_2'$ and recalling the definition of the $Z_M-$norm in \cref{def:ZM}, we have
\begin{align}
    \left| \left\langle \frac{1}{\alpha \beta} r^{2q+1}\de_r (\chi'_2 f'), \Psi \right\rangle\right| & \lesssim \frac{1}{\beta} \|r^{2q+1}\chi'_2\|_{L^\infty} \|f' \|_{L^2([R,+\infty)]}\|\Psi\|_{\dot{H}^1([R,2R])} \\ & \lesssim \frac{ M^{-1/2}}{\beta}  \|f \|_{Z_M([R,+\infty))} \|\Psi\|_{\HH([R,2R])},\label{eq:boundbeta1}
\end{align}
Similarly,
\begin{align}
\left|\left\langle \frac{1}{\alpha \beta} (k_{0}^2-1/4) \chi'_2 r^{2q-1} f, \Psi \right\rangle\right| 
    \lesssim \frac{M^{-1/2}}{\beta} \|f\|_{Z_M([R,+\infty))} \|\Psi\|_{\HH([R,2R])},\label{eq:boundbeta2}
\end{align}
and 
\begin{align}
    \left| \left\langle \left(\frac{1}{2\alpha \beta} + \frac{\alpha -1}{\alpha\beta}\right) r^{2q}(\chi'_2 f'), \Psi \right\rangle\right| \lesssim \frac{M^{-1/2}}{\beta} \|f\|_{Z_M([R,+\infty))} \|\Psi\|_{\HH([R,2R])}
    %
\label{eq:boundbeta3}
\end{align}
The terms that do not contain factors of $\beta$ can be bounded as follows
\begin{align}
|\langle (ikv(r) - \lambda) r^{2q}(\chi'_2 f'), \Psi\rangle| & \lesssim \langle r^{2q}|\chi'_2 f'|, |\Psi| \rangle   \lesssim \|f'\|_{L^2([R,2R])} \|\chi'_2 \Psi\|_{L^2([R,2R])} \\ & \lesssim \|f'\|_{L^2([R,2R])} \|r^{-1} \Psi\|_{L^2([R,2R])}  \\ & \lesssim M^{-1/2} \|f\|_{Z_M([R,+\infty))} \|\Psi\|_{\HH([R,2R])},
\end{align}
and since $(-ikv(r)-\bar c)^{-1}$ does not have singularities in $[R,2R]$ and by the properties of $\tilde T_\beta$, we have
\begin{equation}
\label{eq:psiPsi}
    \| \Psi\|_{\HH([R,2R])} \lesssim \|\psi\|_{H^1_r([0,M])}.
\end{equation}
Therefore, putting all the estimates together, we have proved that when $\beta\gg \eps^{-2}$ we have
\begin{align}
\|\psi\|_{H^1_r} & \lesssim M^{-1/2} \| f\|_{Z_M([R, +\infty))} .
\end{align}
This enables us to conclude that
\begin{equation}
\|\mathcal R_2\|_{Z_M([R, +\infty)) \rightarrow H^1_r([0, M])} = \cO(M^{-1/2}).
\end{equation}
\\
\noindent $\diamond$ \textbf{Bound for $\mathbf{\cM_{2,1}}$}: we now want to bound the operator $\cR_1: H^1_r([0,M]) \to Z_M$. 
Recalling that $\cR_1 = \rmL_{k_0}^{-1}(S_\beta-\lambda_\beta I)^{-1}\mathcal{C}_1$, we choose $f \in H^1_r([0,M]) $ and set $\psi = \cR_1(f)$, so that 
$\rmL_{k_0} \psi = (S_\beta-\lambda_\beta I)^{-1}\mathcal{C}_1(f)$. Similarly to what we did for $\cM_{1,2}$, we get
\begin{align}
\label{eq:M21}
    \langle \rmL_{k_0} \psi, \psi \rangle = \langle \mathcal{C}_1(f) , (-S_\beta-\bar \lambda_\beta I)^{-1} \psi \rangle.
\end{align}
The strategy to bound the several terms is the same as before, but the different bound available on $\chi_1$ from \cref{eq:boundschi1} and the fact that $f\in H^1_r$ plays a key role in the estimates.
The term that requires a special attention is again 
\begin{equation}
    |\langle (ikv(r) - \lambda) (\chi'_1 f'), (-S_\beta-\bar \lambda_\beta I)^{-1}  \psi\rangle|.
\end{equation}
Setting $\Psi = (-S_\beta-\bar \lambda_\beta I)^{-1}  \psi$ and integrating by parts, we see that
\begin{align}
    |\langle (ikv(r) - \lambda_\beta) (\chi'_1 f'), \Psi\rangle|\lesssim \langle |rv' \chi'_1 f|, \frac{1}{r}|\Psi| \rangle|+\langle |(ikv-\lambda_\beta) r\chi''_1 f|, \frac{1}{r}|\Psi| \rangle|+\langle |(ikv(r)-\lambda)\chi'_1 f|, |\Psi '| \rangle|.
\end{align}
Since $|rv'(r)|\lesssim 1$ and $|\chi'|+r|\chi''|\lesssim M^{-1}$, we get
\begin{align}
|\langle (ikv(r) - \lambda)\chi'_1 f', \Psi \rangle | \lesssim   M^{-1} \|f\|_{L^2} \| \Psi\|_{\dot{H}^1}\lesssim M^{-1} \|f\|_{H^1_r} \|\psi\|_{\dot{H}^1}.
\end{align}
Therefore, handling in a similar way the other terms in \cref{eq:M21}, we deduce that 
\begin{align}
    \| \psi \|_{\dot{H}^1([R,+\infty])}^2 \lesssim M^{-1} \| f\|_{H^1_r([0,M])} \|\psi\|_{\dot{H}^1([R,+\infty))}.
\end{align}
By the definition of the $Z_M$ norm in \cref{def:ZM}, this enables us to conclude that
\begin{equation}
\|\mathcal R_1\|_{H^1_r([0,M]) \rightarrow Z_M([R,+\infty))} = \cO(M^{-1/2}).
\end{equation}
\end{proof}

\medskip 

\textbf{Step 5:} we are finally left with solving the problem \cref{eq:projection}. This is equivalent to  
\begin{equation}
    \langle \varphi, \phi_0 \rangle=0
\end{equation}
where the inner product is in the space $L^2([0,M];\dd r)$. With this solvability condition we are able to determine the value of $c$. Indeed, from \cref{eq:innerouter} we have that 
\begin{equation}
    \varphi=\left((I-\cM)^{-1}\begin{pmatrix}
        \cM_{1,1}(\phi_0)\\
        \cM_{2,1}(\phi_0)
    \end{pmatrix}\right)_1=\cM_{1,1}(\phi_0) +\sum_{l= 1}^\infty\left(\cM^l\begin{pmatrix}
        \cM_{1,1}(\phi_0)\\
        \cM_{2,1}(\phi_0)
    \end{pmatrix}\right)_1.
\end{equation}
By the bounds in \cref{prop:boundsM}, we get 
\begin{align}
\label{eq:zero_projection}
    \langle \varphi,\phi_0\rangle_{L^2}=\,&\langle \cM_{1,1}(\phi_0),\phi_0\rangle+\mathcal{O}(\eps^2 M^4+M^{-1})+ M^4\smallO_\beta(1)\\
    \notag =\,&\eps \bigg\langle(I-\cK)^{-1}\rmL_{k_0}^{-1}\left(\frac{\phi_0}{r^2}+\frac{\tilde{c}g'(r)\phi_0}{r(v(r)-c)(v(r)-v(1))}\right),\phi_0\bigg\rangle\\
    \notag &+\big\langle(I-\cK)^{-1}\rmL_{k_0}^{-1}\big(T_\beta\big(ik\frac{g'(r)}{r}\phi_0\big),\phi_0\big\rangle+\mathcal{O}(\eps^2 M^4+M^{-1})+ M^4\smallO_\beta(1).
\end{align}
Then, we observe that $(I-\cK)^{-1}\rmL_{k_0}^{-1}$ is self-adjoint in $L^2(\dd r)$. Indeed, let $f, g\in L^2$ be given and $h=(I-\cK)^{-1}\rmL_{k_0}^{-1}(f)$, $q=(I-\cK)^{-1}\rmL_{k_0}^{-1}(g)$. Then $h$ is the unique solution to $\rmL_{k_0}(I-\cK) h=f$ and analogously for $q$. Since the operator $\rmL_{k_0}(I-\cK)= \rmL_{k_0}-(A(r)+\mathbb{P})$ is self-adjoint, one has 
\begin{equation}
    \langle (I-\cK)^{-1}\rmL_{k_0}^{-1} (f),g\rangle =\langle h,\rmL_{k_0}(I-\cK)(q)\rangle=\langle \rmL_{k_0}(I-\cK) (h),q\rangle=\langle  f,(I-\cK)^{-1}\rmL_{k_0}^{-1}(g)\rangle,
\end{equation}
thus proving that $(I-\cK)^{-1}\rmL_{k_0}^{-1}$ is self-adjoint. Moreover, by the definition of $\phi_0$ and $\mathbb{P}$ in \cref{eq:phi_0} and \cref{eq:proj} respectively, we know that $\mathbb{P}(\phi_0)=\phi_0$ and 
\[
\phi_0=-(I-\cK)^{-1}\rmL_{k_0}^{-1}(\phi_0).
\]
Therefore, recalling that $\|\phi_0/r\|_{L^2}=1$, we can rewrite \cref{eq:zero_projection} as 
\begin{equation}
    \langle \varphi,\phi_0\rangle =-\eps(1+\tilde{c}\Gamma_1(c))+\Gamma_2(\beta)+\mathcal{O}(\eps^2 M^4+M^{-1})+ M^4\smallO_\beta(1)
\end{equation}
where we denote 
\begin{align}
    \Gamma_1(c)=\left\langle \frac{g'(r)}{r(v(r)-c)(v(r)-v(1))}\phi_0,\phi_0 \right\rangle, \qquad \Gamma_2(\beta)=\left\langle T_\beta\big ( ik \frac{g'(r)}{r}\phi_0\big),\phi_0 \right\rangle.
\end{align}
Recalling the definition of $A(r)$ in \cref{def:A} and $h_\eps$ in \cref{lemma:approximation_denominator}, we rewrite $\Gamma_1(c)$  as 
\begin{align}
    \Gamma_1(c)=\, \int_{0}^{M}\frac{1}{v'(1)(r-1)-\eps \tilde c}A(r)|\phi_0|^2(r)\dd r+\int_0^{M}h_\eps(r)A(r)|\phi_0|^2(r)\dd r.
\end{align}
Thanks to \cref{lemma:approximation_denominator} and properties of $\phi_0, A$, the last integral is bounded and its imaginary part converges to zero as $\eps\to 0$. Moreover, since $\phi_0\in C^{0,1/2}([0,R_2])$ by Morrey's inequality and $A(r)\to \gamma_1/v'(1)$ as $r\to 1$, we can invoke the Sochocki-Plemelj formula and deduce that there exists a constant $\kappa$ such that 
\begin{equation}
    \lim_{\eps\to 0} \Gamma_1(c(\eps))= \lim_{\eps\to 0} \Gamma_1(v(1)+\eps \tilde c)=\kappa+i\pi \gamma_1\frac{|\phi_0(1)|^2}{|v'(1)|^2}=:z_0.
\end{equation}
In view of \cref{prop:conv_Sbeta} and the fact that $\phi_0\in H^1_r$, we know that $\Gamma_2(\beta)=\smallO_\beta(1)$.
Therefore
\begin{align}
        \langle \varphi,\phi_0\rangle&=-\eps (\tilde c z_0+1)-\eps \tilde c (\Gamma_1(v(1)+\eps\tilde c)-z_0)+\mathcal{O}(\eps^2 M^4+M^{-1})+ M^4\smallO_\beta(1).
\end{align}
We denote
\begin{align}
f(\tilde c)&=-\eps(\tilde c z_0+1), \\
g(\tilde c, \eps,\beta,M)&=-\eps \tilde c (\Gamma_1(v(1)+\eps\tilde c)-z_0)+\mathcal{O}(\eps^2 M^4+M^{-1})+ M^4\smallO_\beta(1) \\ & 
= \mathcal{O}(\eps^2 M^4+M^{-1})+ M^4\smallO_\beta(1).
\end{align}
Then, we know that  $f(\tilde c)$ has a unique zero in $B_{1/|2z_0|}(- z_0^{-1})$ and this ball is 
 fully contained in the upper half plane since $\Im(-z_0^{-1})>0$ and the radius is sufficiently small. By choosing $\eps$ sufficiently small and $\beta,M$ sufficiently large, one can prove that $|g(\tilde c,\eps,\beta,M)|<|f(\tilde c)|$ for all $\tilde c\in B_{1/|2z_0|}(- z_0^{-1})$. Hence, Rouché's Theorem implies that there exists a unique zero for $f+g(\cdot,\eps,\beta,M)$ in $B_{1/|2z_0|}(z_0^{-1})$, that with a slight abuse in notation we still denote as $\tilde c^*$ . Hence, we have found the desired eigenvalue $c=v(1)+\eps \tilde c$ since  $\Im(c)=\eps \Im(\tilde c^*)>0 $.
\qed 

\section{Expansion of the resolvent of $S_\beta$}
\label{sec:proof_inverse}
In this section, we aim at proving \cref{prop:conv_Sbeta}; we refer to \cref{sec:functionspace} for the definition of the norms, and we simplify the notation by only using $L^2, \HH, H^1_r$.

First of all, we express the resolvent via the Laplace transform of the semigroup associated to the skew-adjoint operator $S_\beta$, namely for all $z$ such that $\mathrm{Re}(z)>0$ we have
\begin{equation}
\label{eq:S_beta_integral}
(S_\beta-zI)^{-1}(w)=-\int_0^{\infty}\e^{-z \tau}\e^{\tau S_\beta }(w) \dd \tau,
\end{equation}
where $w$ is a given function in $L^2$, $H^1$ or $\HH$. We also denote $f_\beta \coloneqq \e^{\tau S_\beta }(w)$, which solves 
\begin{equation}
\label{eq:S_beta}
\begin{cases}
    \partial_\tau f_\beta+ ikv(r)f_\beta-\frac{1}{\alpha\beta}(r\de_r+\frac12)f_\beta = 0\\
    f_\beta|_{\tau = 0} = w.
\end{cases}
\end{equation}
We define $f_\infty \coloneqq\e^{\tau S_\infty}(w)$ where we recall that $S_\infty=-ikv(r)$. The goal of \cref{prop:conv_Sbeta} is to compare $f_\beta$ and $f_\infty$. To this end, observe that $f_\beta, f_\infty$ are explicitly given by 
\begin{equation}
\label{eq:fbetainfty}
    f_\beta(r,\tau)=\e^{\frac{\tau}{2\alpha\beta}-ik\int_0^\tau v\big(\e^{\frac{\tau-s}{\alpha\beta}}r\big)\dd s}w(\e^{\frac{\tau}{\alpha\beta}}r) 
,\qquad f_\infty(r,\tau)=\e^{-ik\tau v(r)}w(r).
\end{equation}
Then, we can split the solution to \cref{eq:S_beta} as 
\begin{equation}
    f_\beta= f_\infty+(f_\beta-f_\infty)
\end{equation}
and rewrite  
\begin{equation}
\label{eq:op_integral}
(S_\beta-z I)^{-1}(w)=-\int_0^{\infty}\e^{-(ikv(r)+z)\tau}(w) \dd \tau-T_{z,\beta}(w)=-(S_\infty-zI)^{-1}(w)-T_{z,\beta}(w),
\end{equation}
where we finally define the operator
\begin{equation}
\label{def:T_beta}
    T_{\beta}(w)\coloneqq T_{z,\beta}(w)\coloneqq\int_0^{\infty}\e^{-z\tau}
({f}_\beta-f_\infty)(\tau)\dd \tau
\end{equation}
Notice that \cref{eq:op_integral} is exactly the expansion stated in \cref{prop:conv_Sbeta}. 
It thus remain to show the bounds we have claimed on $T_{\beta}$, which are a consequence of the pointwise in time estimates of the integrand which we collect below. 
\begin{lemma}
\label{lemma:small_tau}
    Let $f_\beta, f_\infty$ be given as in \cref{eq:fbetainfty} with $w\in X\in \{L^2,\dot{H}^1,H^1_r\}$. Then, for any $T>0$ 
    \begin{equation}
        \lim_{\beta\to \infty}\|(f_\beta-f_\infty)(\tau)\|_{X}=0, \qquad \text{for all } 0\leq \tau \leq T \text{ and } X\in \{L^2,\, \dot{H}^1, \, H^1_r\}.
    \end{equation}
\end{lemma}
\begin{proof}
 We denote 
 \[
 d_\beta[w](r,\tau) \coloneqq (f_\beta-f_\infty)(r,\tau)=\e^{\frac{\tau}{2\alpha\beta}-ik\int_0^\tau v\big(\e^{\frac{\tau-s}{\alpha\beta}}r\big)\dd s}w(\e^{\frac{\tau}{\alpha\beta}}r) 
-\e^{-ik\tau v(r)}w(r).
 \]
    Then $d_\beta[w](r,\tau)\to 0$ pointwise as $\beta\to \infty$ and $|d_\beta[w](r,\tau)|\leq \e^{\frac{\tau}{2\alpha\beta}}|w(\e^{\frac{\tau}{\alpha\beta}}r)|+|w(r)|\coloneqq g_\beta(r,\tau)$. Since $\|g_\beta(\tau)\|_{L^2(dr)}\leq 2\|w\|_{L^2(dr)}$,  by dominated convergence theorem we conclude that if $w\in L^2$ we have $\|d_\beta(\tau)\|_{L^2} \to 0$ as $\beta \to \infty$ uniformly in $\tau$. To control a piece of the $\dot{H}^1$ norm, observe that 
    \begin{equation}
        r^{-1}d_\beta[w](r,\tau)=d_\beta[r^{-1}w](r,\tau)+\e^{\frac{\tau}{2\alpha\beta}-ik\int_0^\tau v\big(\e^{\frac{\tau-s}{\alpha\beta}}r\big)\dd s}\big(\e^{\frac{\tau}{\alpha \beta}}-1\big)(r^{-1}w)(\e^{\frac{\tau}{\alpha\beta}}r) 
    \end{equation}
    We can then argue as before to conclude that for any fixed $T>0$, if $r^{-1}w\in L^2$ then $\|r^{-1}d_\beta(\tau)\|_{L^2}\to 0$ as $\beta \to \infty$ for all $\tau\in [0,T]$. To control the derivative in $r$, notice that
    \begin{align}
        &\de_r d_\beta[w](r,\tau)=d_\beta[\de_rw](r,\tau)+\e^{\frac{\tau}{2\alpha\beta}-ik\int_0^\tau v\big(\e^{\frac{\tau-s}{\alpha\beta}}r\big)\dd s}\big(\e^{\frac{\tau}{\alpha \beta}}-1\big)(\de_rw)(\e^{\frac{\tau}{\alpha\beta}}r) \\
        &+\e^{ik\tau v(r)} ikr\tau v'(r) (r^{-1}w)(r)-ik\Big(\int_0^\tau (rv')(\e^{\frac{\tau-s}{\alpha \beta}}r)\dd s\Big)\e^{\frac{\tau}{2\alpha\beta}-ik\int_0^\tau v\big(\e^{\frac{\tau-s}{\alpha\beta}}r\big)\dd s}r^{-1}w(\e^{\frac{\tau}{\alpha\beta}}r).
    \end{align}
    From the identity above, since $rv'\in L^\infty$, it is not hard to deduce that for any fixed $T>0$,  if $w\in \dot{H}^1$ then $\|\de_rd_\beta(\tau)\|_{L^2}\to0$ as $\beta\to \infty$ for all $\tau\in [0,T]$.

    For the bound of the $H^1_r$ norm, we only have to control the $r\de_rd_\beta$ term, which readily follows by analogous arguments we have presented for $r^{-1}d_\beta$ and $\de_rd_\beta$ and therefore we omit the proof.
\end{proof}
\begin{remark}
    One can easily prove the result in \cref{lemma:small_tau} also for higher derivatives. However, one should not expect to obtain convergence in $\dot{H}^\beta$ for instance. Indeed, if we take a number of derivatives comparable to $\beta$, then we lose a uniform control with respect to $\beta$. This was foreseen, due to the concentration caused by  the $r\de_r$ term.
\end{remark}
With \cref{lemma:small_tau} at hand, we are finally ready to prove \cref{prop:conv_Sbeta}.
\begin{proof}[Proof of \cref{prop:conv_Sbeta}]
Let us consider $w \in X$ with $X\in \{L^2, \dot{H}^1,H^1_r\}$ and $\|w\|_X=1$. By the definition of $T_{\beta}$ given in \ref{def:T_beta}, we have:
\begin{equation}
\|T_{\beta} w\|_{X} \leq \int_0^T e^{-\Re(z)\tau} \| (f_\beta - f_\infty)(\tau)\|_{X} d\tau + \int_T^\infty e^{-\Re(z)\tau}  \big(\| f_\beta(\tau)\|_{X} +\|f_\infty(\tau)\|_{X}\big)d\tau
\end{equation}
for some $T>1$ to be specified later. For the second integral, we use the rough bound
\begin{equation}
    \| f_\beta(\tau)\|_{X} +\|f_\infty(\tau)\|_{X}\lesssim (1+\tau)\e^{\frac{\tau}{\alpha\beta}}.
\end{equation}
Since $T>1 $ and $\Re(z)>0$, we know that for all $\beta>\beta_0$ with $\beta_0>2/(\alpha\Re(z))$ we have
\begin{equation}
    \int_T^\infty e^{-\Re(z)\tau}  \big(\| f_\beta(\tau)\|_{X} +\|f_\infty(\tau)\|_{X}\big)d\tau\lesssim \int_T^{\infty}\tau \e^{-(\Re(z)-\frac{1}{\alpha\beta})\tau}\dd \tau\leq\frac{C}{\Re(z)}\e^{-\frac12\Re(z)T}.
\end{equation}
Hence
\begin{equation}
    \|T_{\beta} w\|_{X} \leq \frac{1}{\Re(z)}\sup_{\tau\in [0,T]}\|(f_\beta-f_\infty)(\tau)\|_X+C\frac{1}{\Re(z)}\e^{-\frac12\Re(z)T}.
\end{equation}
For any $\eps>0$, we choose $T$ large enough so that the second term is of size $\eps/2$ and use \cref{lemma:small_tau} to conclude that there exists a $\beta_0$ sufficiently large such that for all $\beta\geq \beta_0$ we have  $\|T_\beta\|_{X\to X}\leq \eps$, thus proving the desired result.
\end{proof}

\bigskip 

\subsection*{Acknowledgements} The authors would like to D. Albritton and M. Colombo for useful discussions about the result, and the anonymous referees for valuable comments and suggestions. The research of MD and GM was supported by the Swiss State Secretariat for Education, Research and lnnovation (SERI) under contract number MB22.00034 through the project TENSE. MD was supported also by the Swiss National Science Foundation (SNF Ambizione grant PZ00P2\_223294) and is a member of the GNAMPA-INdAM.

\bibliographystyle{siam}
\bibliography{Nonuniqueness2d}

\end{document}

%% file: macros.tex
\newcommand{\red}[1]{\textcolor{red}{#1}}

\def\dd{{\rm d}}

\def\eps{\varepsilon}
\def\e{{\rm e}} 
\def\de{{\partial}}

\def\RR {\mathbb{R}}
\def\TT {\mathbb{T}}
\def\ZZ {\mathbb{Z}}

\def\HH {\dot{H}^1}

\def\Re{{\rm Re}}
\def\Im{{\rm Im}}

\newcommand{\taninv}{\tan^{-1}}

\newcommand{\rmin}{\mathrm{in}}
\newcommand{\rmout}{\mathrm{out}}
\newcommand{\rmL}{\mathrm{L}}

\newcommand{\cM}{\mathcal{M}}

\newcommand{\cK}{\mathcal{K}}

\newcommand{\cR}{\mathcal{R}}

\newcommand{\cO}{\mathcal{O}}
\newcommand\smallO{
  \mathchoice
    {{\scriptstyle\mathcal{O}}}
    {{\scriptstyle\mathcal{O}}}
    {{\scriptscriptstyle\mathcal{O}}}
    {\scalebox{.7}{$\scriptscriptstyle\mathcal{O}$}}
  }

\newtheorem{proposition}{Proposition}[section]
\newtheorem{theorem}[proposition]{Theorem}
\newtheorem{corollary}[proposition]{Corollary}
\newtheorem{lemma}[proposition]{Lemma}

\theoremstyle{definition}
\newtheorem{definition}[proposition]{Definition}

\newtheorem{remark}[proposition]{Remark}

\theoremstyle{definition}

\numberwithin{equation}{section}


%% file: Nonuniqueness2d.bbl
\begin{thebibliography}{10}

\bibitem{albritton22annals}
{\sc D.~Albritton, E.~Bru\'e, and M.~Colombo}, {\em Non-uniqueness of {L}eray
  solutions of the forced {N}avier-{S}tokes equations}, Ann. of Math. (2), 196
  (2022), pp.~415--455.

\bibitem{albritton23Leray}
\leavevmode\vrule height 2pt depth -1.6pt width 23pt, {\em Gluing non-unique
  {N}avier-{S}tokes solutions}, Ann. PDE, 9 (2023), pp.~Paper No. 17, 25.

\bibitem{vishik_notes}
{\sc D.~Albritton, E.~Bru\'e, M.~Colombo, C.~De~Lellis, V.~Giri, M.~Janisch,
  and H.~Kwon}, {\em Instability and non-uniqueness for the 2{D} {E}uler
  equations, after {M}. {V}ishik}, vol.~219 of Annals of Mathematics Studies,
  Princeton University Press, Princeton, NJ, 2024.

\bibitem{albritton23Hypo}
{\sc D.~Albritton and M.~Colombo}, {\em Non-uniqueness of {L}eray solutions to
  the hypodissipative {N}avier-{S}tokes equations in two dimensions}, Comm.
  Math. Phys., 402 (2023), pp.~429--446.

\bibitem{albritton2023linear}
{\sc D.~Albritton and W.~O{\.z}a{\'n}ski}, {\em Linear and nonlinear
  instability of vortex columns}, arXiv:2310.20674,  (2023).

\bibitem{coiculescu2024stability}
{\sc T.~Binz and M.~P. Coiculescu}, {\em Stability of the inviscid power-law
  vortex}, arXiv:2411.13397,  (2024).

\bibitem{bressan20murray}
{\sc A.~Bressan and R.~Murray}, {\em On self-similar solutions to the
  incompressible {E}uler equations}, J. Differential Equations, 269 (2020),
  pp.~5142--5203.

\bibitem{bressan21shen}
{\sc A.~Bressan and W.~Shen}, {\em A posteriori error estimates for
  self-similar solutions to the {E}uler equations}, Discrete Contin. Dyn.
  Syst., 41 (2021), pp.~113--130.

\bibitem{brue2024flexibility}
{\sc E.~Bru{\`e}, M.~Colombo, and A.~Kumar}, {\em {Flexibility of
  Two-Dimensional Euler Flows with Integrable Vorticity}}, arXiv:2408.07934,
  (2024).

\bibitem{castro2024proof}
{\sc A.~Castro, D.~Faraco, F.~Mengual, and M.~Solera}, {\em A proof of
  {V}ishik's nonuniqueness theorem for the forced 2{D} {E}uler equation}, J.
  Reine Angew. Math., 824 (2025), pp.~253--288.

\bibitem{elgindi2023wellposedness}
{\sc T.~M. Elgindi, R.~W. Murray, and A.~R. Said}, {\em {Wellposedness and
  singularity formation beyond the Yudovich class}}, arXiv:2312.17610,  (2023).

\bibitem{elling12spiral}
{\sc V.~Elling}, {\em Existence of algebraic vortex spirals}, in Hyperbolic
  problems---theory, numerics and applications. {V}olume 1, vol.~17 of Ser.
  Contemp. Appl. Math. CAM, World Sci. Publishing, Singapore, 2012,
  pp.~203--214.

\bibitem{elling16spiral}
\leavevmode\vrule height 2pt depth -1.6pt width 23pt, {\em Self-similar 2d
  {E}uler solutions with mixed-sign vorticity}, Comm. Math. Phys., 348 (2016),
  pp.~27--68.

\bibitem{filonov2021nonuniqueness}
{\sc N.~Filonov and P.~Khodunov}, {\em {Nonuniqueness of Leray--Hopf solutions
  for a dyadic model}}, St. Petersburg Mathematical Journal, 32 (2021),
  pp.~371--387.

\bibitem{golovkin1964nonuniqueness}
{\sc K.~K. Golovkin}, {\em Nonuniqueness of the solutions of certain boundary
  problems for the equations of hydromechanics}, USSR Computational Mathematics
  and Mathematical Physics, 4 (1964), pp.~212--215.

\bibitem{golovkin1965examples}
\leavevmode\vrule height 2pt depth -1.6pt width 23pt, {\em Examples of the
  non-uniqueness and low stability of solutions of the equations of
  hydrodynamics}, USSR Computational Mathematics and Mathematical Physics, 5
  (1965), pp.~115--132.

\bibitem{jia2015incompressible}
{\sc H.~Jia and V.~\v{S}ver\'ak}, {\em {Are the incompressible 3d
  Navier--Stokes equations locally ill-posed in the natural energy space?}},
  Journal of Functional Analysis, 268 (2015), pp.~3734--3766.

\bibitem{Kong96}
{\sc Q.~Kong and A.~Zettl}, {\em Eigenvalues of regular {S}turm-{L}iouville
  problems}, J. Differential Equations, 131 (1996), pp.~1--19.

\bibitem{kumar2023simple}
{\sc A.~Kumar and W.~O\.za\'nski}, {\em A simple proof of linear instability of
  shear flows with application to vortex sheets}, J. Math. Fluid Mech., 27
  (2025), pp.~Paper No. 34, 14.

\bibitem{ladyvzenskaja1969example}
{\sc O.~A. Lady{\v{z}}enskaja}, {\em {Example of nonuniqueness in the Hopf
  class of weak solutions for the Navier-Stokes equations}}, Mathematics of the
  USSR-Izvestiya, 3 (1969), p.~229.

\bibitem{ladyzhenskaya1971kirill}
\leavevmode\vrule height 2pt depth -1.6pt width 23pt, {\em {Kirill Kapitonovich
  Golovkin (1936--1969)}}, Trudy Matematicheskogo Instituta imeni VA Steklova,
  116 (1971), pp.~3--4.

\bibitem{lin2003instability}
{\sc Z.~Lin}, {\em Instability of some ideal plane flows}, SIAM journal on
  mathematical analysis, 35 (2003), pp.~318--356.

\bibitem{marchioropulvirenti}
{\sc C.~Marchioro and M.~Pulvirenti}, {\em Mathematical theory of
  incompressible nonviscous fluids}, vol.~96 of Applied Mathematical Sciences,
  Springer-Verlag, New York, 1994.

\bibitem{shao2023self}
{\sc F.~Shao, D.~Wei, and Z.~Zhang}, {\em Self-similar algebraic spiral
  solution of 2-{D} incompressible {E}uler equations}, Ann. PDE, 11 (2025),
  pp.~Paper No. 13, 91.

\bibitem{tollmien1935allgemeines}
{\sc W.~Tollmien}, {\em Ein allgemeines kriterium der instabilitat laminarer
  gescgwindigkeits-verteilungen}, Nachr. Wiss Fachgruppe, G{\"o}ttingen, Math.
  Phys., 1 (1935), p.~79.

\bibitem{vishik2018instability1}
{\sc M.~Vishik}, {\em {Instability and non-uniqueness in the Cauchy problem for
  the Euler equations of an ideal incompressible fluid. Part I}},
  arXiv:1805.09426,  (2018).

\bibitem{vishik2018instability2}
\leavevmode\vrule height 2pt depth -1.6pt width 23pt, {\em {Instability and
  non-uniqueness in the Cauchy problem for the Euler equations of an ideal
  incompressible fluid. Part II}}, arXiv:1805.09440,  (2018).

\bibitem{yudovich1963non}
{\sc V.~I. Yudovich}, {\em Non-stationary flow of an ideal incompressible
  liquid}, USSR Computational Mathematics and Mathematical Physics, 3 (1963),
  pp.~1407--1456.

\end{thebibliography}
